\documentclass[10pt]{amsart} 
\usepackage{amsmath,amssymb,bm, color}

\usepackage{amsmath}
\usepackage{graphicx,float} 

\usepackage{mathrsfs}
\usepackage{bbm}
\usepackage{bm}
\usepackage{amsfonts,amssymb}
\usepackage{multirow}
\usepackage{lineno}
\usepackage{color}

\numberwithin{equation}{section}
 \newtheorem{theorem}{Theorem}[section]
 \newtheorem{lemma}[theorem]{Lemma}

\def\3bar{{|\hspace{-.02in}|\hspace{-.02in}|}}
\def\E{{\mathcal{E}}}
\def\T{{\mathcal{T}}}

\def\beta{\boldsymbol{\eta}}

\def\cal#1{{\mathcal #1}}
\def\pT{{\partial T}}

\def\bw{{\mathbf{w}}}
\def\bf{{\mathbf{f}}}
\def\bu{{\mathbf{u}}}
\def\bv{{\mathbf{v}}}
\def\bn{{\mathbf{n}}}

\def\bzeta{{\boldsymbol{\zeta}}} 
\def\bvarphi{{\boldsymbol{\varphi}}}

\newtheorem{remark}{Remark}[section]
\newtheorem{algorithm}{Auto-Stabilized WG Algorithm}[section]

\numberwithin{equation}{section}

\def\3bar{{|\hspace{-.02in}|\hspace{-.02in}|}}
\def\p#1{\begin{pmatrix}#1\end{pmatrix}}
 \def\cal#1{\mathcal{#1}}
  \def\b#1{\mathbf{#1}} 
 \def\an#1{\begin{align}#1\end{align}}

\begin{document}

\title []
 {Auto-Stabilized Weak Galerkin Finite Element Methods for Stokes Equations on Non-Convex Polytopal Meshes}

  \author {Chunmei Wang}
  \address{Department of Mathematics, University of Florida, Gainesville, FL 32611, USA. }
  \email{chunmei.wang@ufl.edu}
  \thanks{The research of Chunmei Wang was partially supported by National Science Foundation Grant DMS-2136380.}

\author {Shangyou Zhang}
\address{Department of Mathematical Sciences,  University of Delaware, Newark, DE 19716, USA}   \email{szhang@udel.edu}

\begin{abstract}
 This paper introduces an auto-stabilized weak Galerkin (WG) finite element method for solving Stokes equations without relying on traditional stabilizers. The proposed WG method accommodates both convex and non-convex polytopal elements in finite element partitions, leveraging bubble functions as a key analytical tool. The simplified WG method is symmetric and positive definite, and optimal-order error estimates are derived for WG approximations in both the discrete $H^1$ norm and the $L^2$ norm.
\end{abstract}


 
\maketitle

\section{Introduction}  
In this paper, we propose a novel weak Galerkin (WG) finite element method for the Stokes equations, designed to eliminate the need for traditional stabilizers while being applicable to non-convex polygonal or polyhedral meshes. Specifically, we consider the Stokes equations with Dirichlet and Neumann boundary conditions: find the unknown functions $\bu$ and $p$ such that
\begin{equation}\label{model}
 \begin{split}
  -\Delta \bu+\nabla p=&\bf, \qquad\qquad \text{in}\quad \Omega,\\
 \nabla\cdot  \bu=&0,\qquad\qquad \text{in}\quad \Omega,\\
 \bu=&0,\qquad\qquad \text{on}\quad \partial\Omega,
 \end{split}
 \end{equation}
where $\Omega\subset \mathbb R^d$ ($d=2, 3$) is a polygonal or polyhedral domain.

The variational formulation of the Stokes problem \eqref{model} is as follows: Find   $\bu\in [H_0^1(\Omega)]^d$ and $p\in L_0^2(\Omega)$ such that the following equations hold: 
\begin{equation}\label{weak}
 \begin{split}
     (\nabla \bu, \nabla \bv)-(\nabla \cdot\bv, p)=&(\bf, \bv), \qquad \forall \bv\in [H_0^1(\Omega)]^d \\
     (\nabla \cdot\bu, q)=&0, \qquad \qquad\forall q\in L_0^2(\Omega),
 \end{split}
 \end{equation} 
  where  $H_0^1(\Omega)=\{w\in H^1(\Omega): w|_{\partial\Omega}=0\}$ and $L_0^2(\Omega)=\{q\in L^2(\Omega); \int_{\Omega} qdx=0\}$. 

Over the past several decades, finite element methods (FEMs) for solving \eqref{model} have been extensively developed based on the weak formulation \eqref{weak}. These methods employ finite-dimensional subspaces of $[H_0^1(\Omega)]^d\times L_0^2(\Omega)$, composed of piecewise polynomials, and require the construction of finite element spaces that satisfy the inf-sup condition established by Babuska \cite{1} and Brezzi \cite{3}. Although effective, the inf-sup condition imposes significant constraints on element design and mesh generation, reducing flexibility. For specific examples and further details on classical FEMs for the Stokes equations, readers are referred to \cite{6}.
In addition to FEMs, other numerical approaches have been developed for solving the Stokes equations, including finite volume methods (FVMs) \cite{z5} and finite difference methods (FDMs) \cite{z4, z16, z17}. 

The WG finite element method provides a novel framework for the numerical solution of partial differential equations (PDEs). This method approximates differential operators within a framework inspired by distribution theory, particularly for piecewise polynomial functions. Unlike traditional approaches, WG reduces regularity requirements on approximating functions by employing carefully designed stabilizers. Extensive research has demonstrated the versatility and effectiveness of WG methods across a broad spectrum of model PDEs \cite{wg1, wg2, wg3, wg4, wg5, wg6, wg7, wg8, wg9, wg10, wg11, wg12, wg13, wg14, fedi, wg15, wg16, wg17, wg18, wg19, wg20, wg21, itera, wy3655}, establishing them as a powerful tool in scientific computing.

A distinctive feature of WG methods is their innovative use of weak derivatives and weak continuity to develop numerical schemes based on the weak forms of PDEs. This unique structure grants WG methods exceptional flexibility, enabling them to handle a wide variety of PDEs while maintaining stability and accuracy in numerical solutions. WG methods have been specifically proposed for solving the Stokes equations \cite{yewang}, including stabilizer-free WG methods designed for convex polytopal meshes \cite{zhang, zhangstokes2020}.

A significant advancement within the WG framework is the Primal-Dual Weak Galerkin (PDWG) method, which effectively addresses challenges commonly faced by traditional numerical methods \cite{pdwg1, pdwg2, pdwg3, pdwg4, pdwg5, pdwg6, pdwg7, pdwg8, pdwg9, pdwg10, pdwg11, pdwg12, pdwg13, pdwg14, pdwg15}. The PDWG method formulates numerical solutions as a constrained minimization problem, where the constraints reflect the weak formulation of PDEs through the use of weak derivatives. This approach leads to an Euler-Lagrange equation that incorporates both primal variables and dual variables (Lagrange multipliers), resulting in a robust and efficient numerical scheme.

This paper presents a simplified formulation of the WG  finite element method that supports both convex and non-convex polygonal or polyhedral elements in finite element partitions. Recently, this auto-stabilized WG approach has been successfully applied to the Poisson equation \cite{autosecon}, the biharmonic equation \cite{autobihar}, and linear elasticity problems \cite{autoelas}. A key innovation of the proposed method is the elimination of stabilizers through the use of higher-degree polynomials for computing weak gradient and weak divergence operators. This approach preserves the size and global sparsity of the stiffness matrix while significantly reducing the programming complexity compared to traditional stabilizer-dependent methods. By leveraging bubble functions as a critical analytical tool, the method extends WG to non-convex polytopal elements in finite element partitions, marking a significant advancement over existing stabilizer-free WG methods \cite{zhang, zhangstokes2020}, which are limited to convex polytopal meshes.
Theoretical analysis establishes optimal error estimates for the WG approximations in both the discrete 
$H^1$ norm and the 
$L^2$ norm.

 This paper is organized as follows. Section 2 provides a brief overview of the weak gradient and weak divergence operators, along with their discrete counterparts. In Section 3, we present an auto-stabilized WG scheme for Stokes equations that eliminates the need for stabilization terms and is applicable to non-convex polytopal meshes. Section 4 establishes the existence and uniqueness of the solution for the proposed scheme. In Section 5, the error equation for the WG scheme is derived. Section 6 focuses on obtaining error estimates for the numerical approximation in the discrete
$H^1$
 norm, while Section 7 extends the analysis to derive error estimates in the
$L^2$ norm. Finally, Section 8 presents numerical tests to validate the theoretical results.

In this paper, we use standard notations throughout. Let 
 $D$ represent any open, bounded domain with a Lipschitz continuous boundary in $\mathbb{R}^d$. The inner product, semi-norm, and norm in the Sobolev space  $H^s(D)$ for any integer $s\geq0$ are denoted by $(\cdot,\cdot)_{s,D}$, $|\cdot|_{s,D}$ and $\|\cdot\|_{s,D}$ respectively. For simplicity, when the domain $D$ is $\Omega$, the subscript $D$ is omitted. Additionally, when $s=0$, the notations  
$(\cdot,\cdot)_{0,D}$, $|\cdot|_{0,D}$ and $\|\cdot\|_{0,D}$ are further  simplified to $(\cdot,\cdot)_D$, $|\cdot|_D$ and $\|\cdot\|_D$, respectively.

\section{Discrete Weak Gradient and Discrete Weak Divergence} 
This section revisits the definitions of weak gradient and weak divergence operators, as well as their discrete formulations, as introduced in \cite{yewang}.

Consider a polytopal element $T$ with boundary  $\partial T$. A weak function on $T$  is expressed   as  $\bv=\{\bv_0, \bv_b\}$, where $\bv_0\in [L^2(T)]^d$ represents the interior values of  $\bv$, and  $\bv_b\in [L^{2}(\partial T)]^d$ represents its boundary values. Notably,
  $\bv_b$ is treated independently from the trace of $\bv_0$ on $\partial T$.  
 
The set of all weak functions on  $T$, denote by  $W(T)$, is defined as  
 \begin{equation*}\label{2.1}
 W(T)=\{\bv=\{\bv_0, \bv_b\}: \bv_0\in [L^2(T)]^d, \bv_b\in [L^{2}(\partial
 T)]^d\}.
\end{equation*}
 
 The weak gradient $\nabla_w\bv$ is a linear
 operator mapping 
 $W(T)$ to the dual space of $[H^1(T)]^{d\times d}$. For any
 $\bv\in W(T)$, $\nabla_w\bv$ is defined as a bounded linear functional in   the dual space of $[H^1(T)]^{d\times d}$, given by:
 \begin{equation*} 
  (\nabla_w\bv, \bvarphi)_T=-(\bv_0, \nabla \cdot \bvarphi)_T+
  \langle \bv_b,   \bvarphi \cdot \bn \rangle_{\partial T},\quad \forall \bvarphi\in [H^1(T)]^{d\times d},
  \end{equation*}
 where $\bn$, with components $n_i (i=1,\cdots, d)$,  is the unit outward normal direction to $\partial T$.

Similarly, the weak divergence $\nabla_w\cdot \bv$ is a linear
 operator mapping 
 $W(T)$ to the dual space of $H^1(T)$. For  
 $\bv\in W(T)$, $\nabla_w\cdot \bv$ is defined as a bounded linear functional in the dual space of $H^1(T)$, given by:
 \begin{equation*} 
  (\nabla_w\cdot \bv, w)_T=-(\bv_0, \nabla  w)_T+
  \langle \bv_b\cdot\bn,  w \rangle_{\partial T},\quad \forall w\in H^1(T).
  \end{equation*} 
 
 For any non-negative integer $r\ge 0$, let $P_r(T)$ denote the space of
 polynomials on $T$ with total degree at most 
 $r$.   The discrete weak gradient 
 $\nabla_{w, r, T}\bv$ for   $\bv\in W(T)$ is  the unique polynomial  in $[P_r(T)]^{d\times d}$ satisfying
 \begin{equation}\label{2.4}
(\nabla_{w, r, T}\bv, \bvarphi)_T=-(\bv_0, \nabla \cdot \bvarphi)_T+
  \langle \bv_b,   \bvarphi \cdot \bn \rangle_{\partial T},\quad \forall \bvarphi \in [P_r(T)]^{d\times d}.
  \end{equation}
If  $\bv_0\in
 [H^1(T)]^d$ is smooth, applying integration by parts to the first term on the right-hand side  of (\ref{2.4})  yields:
 \begin{equation}\label{2.4new}
(\nabla_{w, r, T}\bv, \bvarphi)_T= (\nabla \bv_0,  \bvarphi)_T+
  \langle \bv_b-\bv_0,   \bvarphi \cdot \bn \rangle_{\partial T}, \quad \forall \bvarphi \in [P_r(T)]^{d\times d}.
  \end{equation}

 The discrete weak divergence
 $\nabla_{w, r, T}\cdot\bv$   for $\bv\in W(T)$ is  the unique polynomial  in $P_r(T)$ satisfying
 \begin{equation}\label{div}
(\nabla_{w, r, T}\cdot \bv, w)_T=-(\bv_0, \nabla  w)_T+
  \langle \bv_b\cdot\bn,  w \rangle_{\partial T},\quad \forall w\in P_r(T).
  \end{equation}
If $\bv_0\in
 [H^1(T)]^d$ is smooth, integrating by parts in the first term on the right-hand side of (\ref{2.4})  yields:
 \begin{equation}\label{divnew}
(\nabla_{w, r, T}\cdot \bv, w)_T= (\nabla   \cdot\bv_0,  w)_T+
  \langle (\bv_b-\bv_0)\cdot\bn,  w \rangle_{\partial T},\quad \forall w\in P_r(T).
  \end{equation}  

\section{Auto-Stabilized Weak Galerkin Algorithms }\label{Section:WGFEM}
 Let ${\cal T}_h$ be a finite element partition of the domain
 $\Omega\subset \mathbb R^d$ into polytopal elements, satisfying the shape regularity condition outlined in  \cite{wy3655}.
Let ${\mathcal E}_h$ represent the set of all edges or faces in
 ${\cal T}_h$, and and let  ${\mathcal E}_h^0={\mathcal E}_h \setminus
 \partial\Omega$ denote the subset of interior edges or faces. For any element $T\in {\cal T}_h$, let $h_T$ be its diameter, and define the mesh size as 
 $h=\max_{T\in {\cal
 T}_h}h_T$.

 Let $k$ be an integer  such that $k\geq 1$. For each element $T\in\T_h$, the local weak finite element space is defined as:
 \begin{equation*}
 V(k,   T)=\{\{\bv_0,\bv_b\}: \bv_0\in [P_k(T)]^d, \bv_b\in [P_{k}(e)]^d, e\subset \partial T\}.    
 \end{equation*}
The global weak finite element space $V_h$ is obtained by assembling the local spaces $V(k, T)$ for all $T\in {\cal T}_h$ and enforcing the continuity of $\bv_b$ along interior edges or faces  $\E_h^0$: 
 $$
 V_h=\big\{\{\bv_0,\bv_b\}:\ \{\bv_0,\bv_b\}|_T\in V(k, T),
 \forall T\in {\cal T}_h \big\}.
 $$ 
The subspace of $V_h$ consisting of functions with zero boundary values on $\partial\Omega$ is defined as:
$$
V_h^0=\{\bv\in V_h: \bv_b|_{\partial\Omega}=0\}.
$$

For the pressure variable, the weak Galerkin finite element space is defined as: 
$$
W_h=\{q\in L_0^2(\Omega): q|_T \in P_{k-1}(T)\}.
$$

For simplicity, the discrete weak gradient  $\nabla_{w} \bv$  and the discrete weak divergence $\nabla_{w} \cdot\bv$ refer to the operators 
$\nabla_{w, r, T}\bv$ and $\nabla_{w, r, T} \cdot\bv$, respectively, as defined in (\ref{2.4}) and (\ref{div}) for each $T\in {\cal T}_h$:
$$
(\nabla_{w} \bv)|_T= \nabla_{w, r, T}(\bv |_T), \qquad \forall T\in \T_h, 
$$
$$
(\nabla_{w}\cdot \bv)|_T= \nabla_{w, r, T}\cdot (\bv |_T), \qquad \forall T\in \T_h.
$$
 
On each element $T\in\T_h$, let $Q_0$  denote the $L^2$ projection onto $P_k(T)$. On each edge or face  $e\subset\partial T$, let $Q_b$ denote the $L^2$ projection onto $P_{k}(e)$. For any $\bv\in [H^1(\Omega)]^d$,  the $L^2$ projection  into the weak finite element space $V_h$ is  defined as:
 $$
  (Q_h \bv)|_T:=\{Q_0(\bv|_T),Q_b(\bv|_{\pT})\},\qquad \forall T\in\T_h.
$$

The  simplified WG numerical scheme, free from stabilization terms, for solving the Stokes equations \eqref{model} is given as follows:
\begin{algorithm}\label{PDWG1}
Find $\bu_h=\{\bu_0, \bu_b\} \in V_h^0$ and $p_h\in W_h$  such that  
\begin{equation}\label{WG}
\begin{split}
  (\nabla_w \bu_h, \nabla_w \bv_h)-(\nabla_w \cdot\bv_h, p_h)=&(\bf, \bv_0), \qquad \forall \bv_h\in V_h^0,\\
     (\nabla_w \cdot\bu_h, q_h)=&0, \qquad\qquad \forall q_h\in W_h, 
\end{split}
\end{equation}
where 
$$
(\cdot, \cdot)=\sum_{T\in {\cal T}_h}(\cdot, \cdot)_T.
$$
\end{algorithm}

\section{Solution Existence and Uniqueness}  
Recall that ${\cal T}_h$ is a shape-regular finite element partition of the domain $\Omega$. For any  $T\in {\cal T}_h$ and $\phi\in H^1(T)$, the following trace inequality holds \cite{wy3655}: 
\begin{equation}\label{tracein}
 \|\phi\|^2_{\partial T} \leq C(h_T^{-1}\|\phi\|_T^2+h_T \|\nabla \phi\|_T^2).
\end{equation}
If $\phi$ is a polynomial on  $T$,  a simplified trace inequality applies \cite{wy3655}: 
\begin{equation}\label{trace}
\|\phi\|^2_{\partial T} \leq Ch_T^{-1}\|\phi\|_T^2.
\end{equation}

For any $\bv=\{\bv_0, 
\bv_b\}\in V_h$, define the norm: 
\begin{equation}\label{3norm}
\3bar \bv\3bar= (\nabla_{w} \bv, \nabla_{ w} \bv) ^{\frac{1}{2}},
\end{equation}
and  the discrete  $H^1$- semi-norm: 
\begin{equation}\label{disnorm}
\|\bv\|_{1, h}=\Big( \sum_{T\in {\cal T}_h} \|\nabla \bv_0\|_T^2+h_T^{-1}\|\bv_0-\bv_b\|_{\partial T}^2 \Big)^{\frac{1}{2}}.
\end{equation}

\begin{lemma}\label{norm1}
 For $\bv=\{\bv_0, \bv_b\}\in V_h$, there exists a constant $C$ such that
 $$
 \|\nabla 
 \bv_0\|_T\leq C\|\nabla_w \bv\|_T.
 $$
\end{lemma}
\begin{proof} Let  $T\in {\cal T}_h$ be a polytopal element with $N$ edges or faces,  denoted $e_1, \cdots, e_N$. Note that $T$ is not required to be convex. For each edge  or face    $e_i$, define a linear equation $l_i(x)$ such that  $l_i(x)=0$ on $e_i$ as follows: 
$$l_i(x)=\frac{1}{h_T}\overrightarrow{AX}\cdot \bn_i, $$  where  $A$ is a given point on $e_i$, $X$ is any point on  $e_i$, $\bn_i$ is the unit outward normal direction to $e_i$, and $h_T$ is the  size of the element $T$.

The bubble function $\Phi_B$ of  the element  $T$ is  defined as 
 $$
 \Phi_B =l^2_1(x)l^2_2(x)\cdots l^2_N(x) \in P_{2N}(T).
 $$ 
 It is straightforward to verify that  $\Phi_B=0$ on  $\partial T$.    The bubble  function 
  $\Phi_B$  can be scaled such that $\Phi_B(M)=1$ where   $M$ is the barycenter of the element $T$. Furthermore,  there exists a sub-domain $\hat{T}\subset T$ such that $\Phi_B\geq \rho_0$ for some constant $\rho_0>0$.

Given  $\bv=\{\bv_0, \bv_b\}\in V_h$, let  $r=2N+k-1$ and
$\bvarphi=\Phi_B \nabla \bv_0\in [P_r(T)]^{d\times d}$.  Substituting $\bvarphi$ into \eqref{2.4new}, we have
\begin{equation}\label{t1}
\begin{split}
(\nabla_w \bv, \Phi_B \nabla \bv_0)_T&=(\nabla \bv_0, \Phi_B \nabla \bv_0)_T+\langle \bv_b-\bv_0,  \Phi_B \nabla \bv_0 \cdot \bn\rangle_{\partial T}\\&=(\nabla \bv_0, \Phi_B \nabla \bv_0)_T,
\end{split}
\end{equation}
where the second equality follows from the fact that  $\Phi_B=0$ on $\partial T$.

From the domain inverse inequality \cite{wy3655},  there exists a constant $C$ such that 
\begin{equation}\label{t2}
(\nabla \bv_0, \Phi_B \nabla \bv_0)_T \geq C (\nabla \bv_0, \nabla \bv_0)_T.
\end{equation} 
Applying the  Cauchy-Schwarz inequality and combining \eqref{t1} and \eqref{t2}, we have
 $$
 (\nabla \bv_0, \nabla \bv_0)_T\leq C (\nabla_w \bv, \Phi_B \nabla \bv_0)_T  \leq C  \|\nabla_w \bv\|_T \|\Phi_B \nabla \bv_0\|_T  \leq C
\|\nabla_w \bv\|_T \|\nabla \bv_0\|_T,
 $$
which gives
 $$
 \|\nabla \bv_0\|_T\leq C\|\nabla_w \bv\|_T.
 $$

This completes the proof of the lemma.
\end{proof}
 
\begin{remark}
   If the polytopal element $T$  is convex, 
  the bubble function used in Lemma \ref{norm1} can be simplified to 
 $$
 \Phi_B =l_1(x)l_2(x)\cdots l_N(x).
 $$ 
It can be verified that there exists a sub-domain $\hat{T}\subset T$,  such that
 $ \Phi_B\geq\rho_0$  for some constant $\rho_0>0$,  and $\Phi_B=0$ on the boundary $\partial T$.   Using this simplified construction, Lemma \ref{norm1} can be proved in the same manner. In this case, we set $r=N+k-1$.  
\end{remark}

Recall that $T$ is a $d$-dimensional polytopal element and  $e_i$ is a $d-1$-dimensional edge or face  of $T$. 
We define an edge/face-based bubble function as  $$\varphi_{e_i}= \Pi_{k=1, \cdots, N, k\neq i}l_k^2(x).$$ It can be verified that  (1) $\varphi_{e_i}=0$ on the edge/face $e_k$ for $k \neq i$; (2) there exists a subdomain $\widehat{e_i}\subset e_i$ such that $\varphi_{e_i}\geq \rho_1$ for some constant $\rho_1>0$.

\begin{lemma}\label{phi}
     For $\bv=\{\bv_0, \bv_b\}\in V_h$, let $\bvarphi=(\bv_b-\bv_0)^T \bn\varphi_{e_i}$, where $\bn$ is the unit outward normal direction to the edge/face  $e_i$. Then, the following inequality holds:
\begin{equation}
  \|\bvarphi\|_T ^2 \leq Ch_T \int_{e_i}|\bv_b-\bv_0|^2ds.
\end{equation}
\end{lemma}
\begin{proof}
We begin by extending $\bv_b$, initially defined on the $(d-1)$-dimensional edge/face  $e_i$, to the entire $d$-dimensional polytopal element $T$. This extension is defined as:
$$
 \bv_b (X)= \bv_b(Proj_{e_i} (X)),
$$
where $X$ is any point in $T$, $Proj_{e_i} (X)$ is the orthogonal projection of $X$ onto  the hyperplane $H\subset\mathbb R^d$  containing $e_i$. For points where  $Proj_{e_i} (X) \notin e_i$, $\bv_b(Proj_{e_i} (X))$ is defined as the extension of $\bv_b$ from $e_i$ to  $H$.

We claim that $\bv_b$ remains  a polynomial   on  $T$ after the extension.  

Let   $H$ be the hyperplane containing  $e_i$,  defined by $d-1$ linearly independent vectors $\beta_1, \cdots, \beta_{d-1}$ originating from a point $A$ on $e_i$. Any point $P$ on   $H$ can be parametrized as
$$
P(t_1, \cdots, t_{d-1})=A+t_1\beta_1+\cdots+t_{d-1}\beta_{d-1},
$$
where $t_1, \cdots, t_{d-1}$ are parameters.

Since  $\bv_b(P(t_1, \cdots, t_{d-1}))$ is a polynomial of degree $q$ defined on $e_i$, each component $(\bv_b)_k$ of $\bv_b=((\bv_b)_1, \cdots, (\bv_b)_d)$ can be expressed as:
$$
(\bv_b)_k(P(t_1, \cdots, t_{d-1}))=\sum_{|\alpha|\leq q}c_{\alpha}\textbf{t}^{\alpha}, \qquad\forall k=1,\cdots, d,
$$
where $\textbf{t}^{\alpha}=t_1^{\alpha_1}\cdots t_{d-1}^{\alpha_{d-1}}$ and $\alpha=(\alpha_1, \cdots, \alpha_{d-1})$  is a multi-index.

For any point $X$ in  $T$, the projection of   $X$  onto  the hyperplane $H\subset\mathbb R^d$, which contains the edge/face $e_i$, is the point on  $H$ that minimizes the distance to  $X$. Mathematically, this projection $Proj_{e_i} (X)$ is an affine transformation  expressed as 
$$
Proj_{e_i} (X)=A+\sum_{i=1}^{d-1} t_i(X)\beta_i,
$$
where $t_i(X)$ are the projection coefficients, and $A$ is the origin point on $e_i$ and $\beta_i$ are linearly independent vectors defining $H$. The coefficients $t_i(X)$ are determined  by solving the orthogonality condition:
$$
(X-Proj_{e_i} (X))\cdot \beta_j=0, \forall j=1, \cdots, d-1.
$$
This results in a system of linear equations in $t_1(X)$, $\cdots$, $t_{d-1}(X)$, which  can be solved to yield:
$$
t_i(X)= \text{linear function of} \  X.
$$
Hence, the projection $Proj_{e_i} (X)$ is an affine linear function  of $X$.

We extend the polynomial $\bv_b$, originally defined on the edge/face   $e_i$, to the entire element $T$ by defining
$$
(\bv_b)_k(X)=(\bv_b)_k(Proj_{e_i} (X))=\sum_{|\alpha|\leq q}c_{\alpha}\textbf{t}(X)^{\alpha}, \qquad \forall k=1, \cdots, d
$$
where $\textbf{t}(X)^{\alpha}=t_1(X)^{\alpha_1}\cdots t_{d-1}(X)^{\alpha_{d-1}}$ and $\alpha=(\alpha_1, \cdots, \alpha_{d-1})$. Since $t_i(X)$ are linear functions of $X$, each term $\textbf{t}(X)^{\alpha}$ is a polynomial in $X=(x_1, \cdots, x_d)$.
Thus, $\bv_b(X)$ is a polynomial in the $d$-dimensional coordinates $X=(x_1, \cdots, x_d)$.

Similarly, let $\bv_{trace}$ denote the trace of $\bv_0$ on the edge/face  $e_i$. We extend $\bv_{trace}$   to the entire element $T$  using  the  formula:
$$
 \bv_{trace} (X)= \bv_{trace}(Proj_{e_i} (X)),
$$
where $X$ is any point in $T$, $Proj_{e_i} (X)$ denotes the projection of  $X$ onto the hyperplane $H$ containing  $e_i$. When  $Proj_{e_i} (X)$ is not on  $e_i$, $\bv_{trace}(Proj_{e_i} (X))$ is defined as the extension of $\bv_{trace}$ from $e_i$ to the hyperplane $H$.
As in the case of $\bv_b$, $\bv_{trace}$ remains a polynomial after this extension. 

Let $\bvarphi=(\bv_b-\bv_0)^T \bn\varphi_{e_i}$, where $\bn$ the unit outward normal vector to $e_i$. We have
\begin{equation*}
    \begin{split}
\|\bvarphi\|^2_T  =
\int_T \bvarphi^2dT =  &\int_T ((\bv_b-\bv_{trace}) ^T \bn\varphi_{e_i})^2dT\\
\leq &Ch_T \int_{e_i} ((\bv_b-\bv_{trace})^T  \bn\varphi_{e_i})^2ds\\
\leq &Ch_T \int_{e_i} |\bv_b-\bv_0|^2ds, 
    \end{split}
\end{equation*} 
where we used the following facts: (1) $\varphi_{e_i}=0$ on the edge/face $e_k$ for $k \neq i$; (2) there exists a subdomain $\widehat{e_i}\subset e_i$ such that $\varphi_{e_i}\geq \rho_1$ for some constant $\rho_1>0$, and applied the properties of the projection.

 This completes the proof of the lemma.

\end{proof}

\begin{remark}
Consider any $d$-dimensional polytopal element $T$. 
  There exists a hyperplane $H\subset R^d$  such that a finite number $l$ of distinct $(d-1)$-dimensional edges or faces, including $e_{i}$,   are completely contained within $H$. 
 For simplicity,   we assume $l=2$: there exists a   hyperplane $H\subset R^d$    such that two distinct $(d-1)$-dimensional edges or faces,  $e_i$ and $e_m$, are completely contained within $H$.
 The edge/face-based bubble function  is then  defined as
$$\varphi_{e_i}= \Pi_{k=1, \cdots, N, k\neq i, m}l_k^2(x).$$ 
It is easy to check $\varphi_{e_i} = 0$ on   $e_k$ for $k \neq i, m$.

In this case,  Lemma \ref{phi}  requires modification to handle the scenario. Specifically, the final part of its proof is revised as follows.

Let $\sigma$ be a polynomial function satisfying $\sigma \geq \alpha$ on $e_i$ for some constant $\alpha > 0$ and $|\sigma| \leq \epsilon$ on $e_m$ for a sufficiently small $\epsilon > 0$. Define $\bvarphi = (\bv_b - \bv_0)^T \bn \varphi_{e_i} \sigma$. Then:
\begin{equation*}
    \begin{split}
        \|\bvarphi\|^2_T =& \int_T \bvarphi^2 \, dT \\
        =& \int_T \left( (\bv_b - \bv_{trace}) ^T\bn \varphi_{e_i} \sigma \right)^2 \, dT \\
        \leq & C h_T \int_{e_i} \left( (\bv_b - \bv_{trace})^T  \bn \varphi_{e_i} \sigma \right)^2  ds \\
             & + C h_T \int_{e_m}  \left( (\tilde{\bv}_b - \tilde{\bv}_{trace}) ^T\bn \varphi_{e_i} \sigma \right)^2  ds \\
        \leq & C h_T \|(\bv_b - \bv_0)     \|^2_{e_i} 
              + C h_T \epsilon^2 \|(\tilde{\bv}_b - \tilde{\bv}_0)   \|^2_{e_m} \\
        \leq & C h_T \|(\bv_b - \bv_0)    \|^2_{e_i} 
              + C h_T \epsilon^2 \|(\tilde{\bv}_b - \tilde{\bv}_0)  \|^2_{e_i \cup e_m} \\
        \leq & C h_T \|(\bv_b - \bv_0)  \|^2_{e_i} 
              + C h_T \epsilon^2 \|(\bv_b - \bv_0) \|^2_{e_i  } \\
        \leq & C h_T \|(\bv_b - \bv_0)  \|^2_{e_i}.
    \end{split}
\end{equation*} 
Here,  $\tilde{\bv}_b$ and $\tilde{\bv}_{trace}$ denote extensions of $\bv_b$ and $\bv_{trace}$ from $e_i$ to $e_m$. The derivation uses the following key points: 
(1) $\varphi_{e_i} = 0$ on $e_k$ for $k \neq i, m$;
(2) There exists a subdomain $\widehat{e}_i \subset e_i$ such that $\varphi_{e_i} \geq \rho_1$ for some constant $\rho_1 > 0$; 
(3)  The domain inverse inequality;
(4)  The properties of the projection operators.
\end{remark}

\begin{lemma}\label{normeqva}   There exist  positive constants $C_1$ and $C_2$ such that for any $\bv=\{\bv_0, \bv_b\} \in V_h$,  the following norm equivalence holds: 
 \begin{equation}\label{normeq}
 C_1\|\bv\|_{1, h}\leq \3bar \bv\3bar  \leq C_2\|\bv\|_{1, h}.
\end{equation}
\end{lemma}

\begin{proof}   
 Let $T$ be a polytopal element, which may be non-convex. Recall that an edge/face-based bubble function is defined as $$\varphi_{e_i}= \Pi_{k=1, \cdots, N, k\neq i}l_k^2(x).$$

We first extend $\bv_b$ from the edge/face $e_i$ to the element $T$. 
Let $\bv_{trace}$ denote the trace of $\bv_0$ on $e_i$ and extend $\bv_{trace}$ to $T$. For simplicity, these extensions are denoted as  $\bv_b$ and $\bv_0$. Details of the extensions can be found in Lemma \ref{phi}.

Choosing $\bvarphi=(\bv_b-\bv_0)^T \bn\varphi_{e_i}$ in \eqref{2.4new}, we obtain 
\begin{equation}\label{t3} 
 \begin{split}
  (\nabla_{w} \bv, \bvarphi)_T=&(\nabla \bv_0, \bvarphi)_T+
  {  \langle \bv_b-\bv_0,  \bvarphi\cdot \bn \rangle_{\partial T}} \\=&(\nabla \bv_0, \bvarphi)_T+ \int_{e_i} |\bv_b-\bv_0|^2 \varphi_{e_i}ds,   
 \end{split}
  \end{equation} 
   where we used (1) $\varphi_{e_i}=0$ on  $e_k$ for $k \neq i$, 
(2) there exists a subdomain $\widehat{e_i}\subset e_i$ such that $\varphi_{e_i}\geq \rho_1$ for some constant $\rho_1>0$.

 Applying the  Cauchy-Schwarz inequality, \eqref{t3}, the domain inverse inequality \cite{wy3655}, and  Lemma \ref{phi}, we obtain
\begin{equation*}
\begin{split}
 \int_{e_i}|\bv_b-\bv_0|^2  ds\leq &C  \int_{e_i}|\bv_b-\bv_0|^2  \varphi_{e_i}ds \\
 \leq & C(\|\nabla_w \bv\|_T+\|\nabla \bv_0\|_T)\| \bvarphi\|_T\\
 \leq & {Ch_T^{\frac{1}{2}} (\|\nabla_w \bv\|_T+\|\nabla \bv_0\|_T) (\int_{e_i}(|\bv_0-\bv_b|^2ds)^{\frac{1}{2}}}.
 \end{split}
\end{equation*}
Using  Lemma \ref{norm1}, we derive
$$
 h_T^{-1}\int_{e_i}|\bv_b-\bv_0|^2  ds \leq C  (\|\nabla_w \bv\|^2_T+\|\nabla \bv_0\|^2_T)\leq C\|\nabla_w \bv\|^2_T.
$$
Thus, combining this with Lemma \ref{norm1},  \eqref{3norm}, and \eqref{disnorm}, we establish
$$
 C_1\|\bv\|_{1, h}\leq \3bar \bv\3bar.
$$

Next, using \eqref{2.4new}, the Cauchy-Schwarz inequality,  and  the trace inequality \eqref{trace}, we have
$$
 \Big|(\nabla_{w} \bv, \bvarphi)_T\Big| \leq \|\nabla \bv_0\|_T \|  \bvarphi\|_T+
Ch_T^{-\frac{1}{2}}\|\bv_b-\bv_0\|_{\partial T} \| \bvarphi\|_{T},
$$
which implies
$$
\| \nabla_{w} \bv\|_T^2\leq C( \|\nabla \bv_0\|^2_T  +
 h_T^{-1}\|\bv_b-\bv_0\|^2_{\partial T}),
$$
 and thus $$ \3bar \bv\3bar  \leq C_2\|\bv\|_{1, h}.$$

 This completes the proof.
 \end{proof}

  \begin{remark}
   If the polytopal element $T$ is convex, 
  the edge/face-based bubble function in Lemma \ref{normeqva} simplifies to
$$\varphi_{e_i}= \Pi_{k=1, \cdots, N, k\neq i}l_k(x).$$
It can be verified that: (1)  $\varphi_{e_i}=0$ on $e_k$ for $k \neq i$; (2) there exists a subdomain $\widehat{e_i}\subset e_i$ such that $\varphi_{e_i}\geq \rho_1$ for some constant $\rho_1>0$. 

Lemma \ref{normeqva} can be proved in the same manner using this simplified construction. 
\end{remark}

\begin{remark}
Consider a  $d$-dimensional polytopal element $T$. 
  There exists a hyperplane $H\subset R^d$  such that a finite number $l$ of distinct $(d-1)$-dimensional edges or faces containing $e_{i}$ are completely contained within $H$. 
 For simplicity,   we  focus on the case $l=2$: there exists a   hyperplane $H\subset R^d$    such that distinct $(d-1)$-dimensional edges or faces $e_i$ and $e_m$ are  fully  contained within $H$.
 The edge/face-based bubble function is then defined as
$$\varphi_{e_i}= \Pi_{k=1, \cdots, N, k\neq i, m}l_k^2(x),$$ 
where  $\varphi_{e_i} = 0$ on edges or faces $e_k$ for $k \neq i, m$. 
 
For this scenario, Lemma \ref{normeqva} requires a revision. Below is the updated proof for the relevant part of the lemma.

 Let $\sigma$ be a polynomial function such that $\sigma \geq \alpha$ on $e_i$ for some constant $\alpha>0$ and $|\sigma| \leq \epsilon$ on $e_m$ for a sufficiently small $\epsilon>0$. Substituting $\bvarphi=(\bv_b-\bv_0)^T \bn\varphi_{e_i}\sigma$ into  \eqref{2.4new} yields
 \begin{equation}\label{t3-1}  
     \begin{split}
    &  (\nabla_{w} \bv, \bvarphi)_T\\ =&(\nabla \bv_0, \bvarphi)_T+
  {  \langle \bv_b-\bv_0,  \bvarphi\cdot \bn \rangle_{\partial T}} \\ =&(\nabla \bv_0, \bvarphi)_T+ \int_{e_{i}}|\bv_b-\bv_0|^2  \varphi_{e_i}\sigma ds+\int_{e_{m}}(\bv_b-\bv_0) (\tilde{\bv}_b-\tilde{\bv}_0)  \varphi_{e_i}\sigma ds\\
  \geq &(\nabla \bv_0, \bvarphi)_T+ \int_{e_{i}}|\bv_b-\bv_0|^2  \varphi_{e_i}\alpha  ds+\int_{e_{m}}(\bv_b-\bv_0)  (\tilde{\bv}_b-\tilde{\bv}_0)  \varphi_{e_i}\sigma ds, 
     \end{split}
  \end{equation}
   where 
  $\tilde{\bv}_b$ and $\tilde{\bv}_0$ are  extensions of $\bv_b$ and $\bv_0$  from  $e_i$ to $e_m$.

Using the Cauchy-Schwarz inequality, \eqref{t3-1}, the domain inverse inequality \cite{wy3655}, and  Lemma \ref{phi}, we obtain
\begin{equation}\label{s1}
\begin{split}
 &\int_{e_i}|\bv_b-\bv_0|^2  ds\\\leq &C  \int_{e_i}|\bv_b-\bv_0|^2  \varphi_{e_i}\alpha ds \\
 \leq & C(\|\nabla_w \bv\|_T+\|\nabla \bv_0\|_T)\| \bvarphi\|_T +  \Big|\int_{e_{m}}(\bv_b-\bv_0)  (\tilde{\bv}_b-\tilde{\bv}_0)  \varphi_{e_i}\sigma ds\Big|\\
 \leq & Ch_T^{\frac{1}{2}} (\|\nabla_w \bv\|_T+\|\nabla v_0\|_T)  \|\bv_0-\bv_b\|_{e_i} +  C\epsilon \|\bv_b-\bv_0\|_{e_m} \| \tilde{\bv}_b-\tilde{\bv}_0\|_{e_m}\\
 \leq & Ch_T^{\frac{1}{2}} (\|\nabla_w \bv\|_T+\|\nabla v_0\|_T)  \|\bv_0-\bv_b\|_{e_i} +  C\epsilon \|\bv_b-\bv_0\|_{e_m} \| \tilde{\bv}_b-\tilde{\bv}_0\|_{e_m\cup e_i}\\
  \leq & Ch_T^{\frac{1}{2}} (\|\nabla_w \bv\|_T+\|\nabla \bv_0\|_T)  \|\bv_0-\bv_b\|_{e_i} +  C\epsilon \|\bv_b-\bv_0\|_{e_m} \| \bv_b-\bv_0\|_{e_i}.
 \end{split}
\end{equation}
This implies
$$
\| \bv_b-\bv_0\|_{e_i} \leq Ch_T^{\frac{1}{2}} (\|\nabla_w \bv\|_T+\|\nabla \bv_0\|_T)   +  C\epsilon \|\bv_b-\bv_0\|_{e_m}.
$$

 Analogous to \eqref{s1}, we derive
\begin{equation}\label{s2}
    \| \bv_b-\bv_0\|_{e_m} \leq Ch_T^{\frac{1}{2}} (\|\nabla_w \bv\|_T+\|\nabla \bv_0\|_T)   +  C\epsilon \|\bv_b-\bv_0\|_{e_i}.
\end{equation}

Substituting \eqref{s2} into \eqref{s1}, we get
$$
\| \bv_b-\bv_0\|_{e_i} \leq Ch_T^{\frac{1}{2}} (1+\epsilon)(\|\nabla_w \bv\|_T+\|\nabla \bv_0\|_T)   + C \epsilon^2 \|\bv_b-\bv_0\|_{e_i}. 
$$
Reorganizing terms gives
$$
(1-\epsilon^2)\| \bv_b-\bv_0\|_{e_i} \leq Ch_T^{\frac{1}{2}} (1+\epsilon)(\|\nabla_w \bv\|_T+\|\nabla \bv_0\|_T). 
$$
Thus, we have
$$
\| \bv_b-\bv_0\|_{e_i} \leq Ch_T^{\frac{1}{2}} (\|\nabla_w \bv\|_T+\|\nabla \bv_0\|_T). 
$$
 \end{remark}

 \begin{lemma} \cite{yewang}
     There exists a constant $\alpha>0$,  independent of $h$,  such that for all $\rho\in W_h$, 
     \begin{equation}\label{infsup}
         \sup_{\bv\in V_h^0} \frac{(\nabla_w \cdot \bv, \rho)}{\3bar \bv\3bar}\geq \alpha \|\rho\|.
     \end{equation}
 \end{lemma}
\begin{theorem}
The  Auto-Stabilized WG Algorithm  \ref{PDWG1} has  a unique solution. 
\end{theorem}
\begin{proof}
Assume that $(\bu_h^{(1)}, p_h^{(1)})\in V_h^0\times W_h$ and $(\bu_h^{(2)}, p_h^{(2)})\in V_h^0\times W_h$  are two  distinct solutions of the Auto-Stabilized WG scheme \ref{PDWG1}. Define  ${\cal D}_{\bu_h}= \{{\cal D}_{\bu_0}, {\cal D}_{\bu_b}\}=\bu_h^{(1)}-\bu_h^{(2)}\in V_h^0$ and ${\cal D}_{p_h}=p_h^{(1)}-p_h^{(2)}\in W_h$. The pair ${\cal D}_{\bu_h}$  and ${\cal D}_{p_h}$ satisfies
\begin{equation}\label{wgstokes}
\begin{split}
  (\nabla_w {\cal D}_{\bu_h}, \nabla_w \bv_h)-(\nabla_w \cdot\bv_h, {\cal D}_{p_h})=&0, \qquad \qquad \forall \bv_h\in V_h^0,\\
     (\nabla_w \cdot {\cal D}_{\bu_h}, q_h)=&0, \qquad\qquad \forall q_h\in W_h.
\end{split}
\end{equation}
Choosing $\bv_h={\cal D}_{\bu_h}$ and $q_h={\cal D}_{p_h}$ in \eqref{wgstokes} yields $\3bar {\cal D}_{\bu_h} \3bar=0$. From the equivalence of the norm in \eqref{normeq},   it follows that $\|{\cal D}_{\bu_h}\|_{1,h}=0$, which  implies:
1) $\nabla {\cal D}_{\bu_0}=0$ on each $T$,  and 2) ${\cal D}_{\bu_0}={\cal D}_{\bu_b}$ on each $\partial T$.  Since  $\nabla {\cal D}_{\bu_0}=0$ on each $T$,  
${\cal D}_{\bu_0}$ is a constant within each $T$.  Combining this with ${\cal D}_{\bu_0}={\cal D}_{\bu_b}$ on $\partial T$, we conclude that     ${\cal D}_{\bu_0}$  is a constant throughout the domain $\Omega$. Furthermore, since ${\cal D}_{\bu_0}=0$ on $\partial\Omega$,  it follows that ${\cal D}_{\bu_0}\equiv 0$ in   $\Omega$. Using ${\cal D}_{\bu_0}={\cal D}_{\bu_b}$ on $\partial T$, we deduce ${\cal D}_{\bu_b}\equiv 0$, and thus ${\cal D}_{\bu_h}\equiv 0$ in $\Omega$. 
  Substituting ${\cal D}_{\bu_h}\equiv 0$ into the first equation of 
\eqref{wgstokes} gives 
$$
(\nabla_w\cdot \bv_h, {\cal D}_{p_h})=0
$$
for any $\bv_h\in V_h$. By the inf-sup condition \eqref{infsup}, this implies 
$\|{\cal D}_{p_h}\|=0$, and hence   $ {\cal D}_{p_h} \equiv  0$ in $\Omega$.

Thus, we conclude that
 $\bu_h^{(1)}\equiv \bu_h^{(2)}$ and $p_h^{(1)}\equiv p_h^{(2)}$, proving the uniqueness of the solution.
\end{proof}

\section{Error Equations}
Let ${\cal Q}_h$ denote the $L^2$ projection operator onto the finite element space of piecewise polynomials of degree at most $2N+k-1$ for non-convex element and $N+k-1$ for convex element in the finite element partitions.  

\begin{lemma}\label{Lemma5.1}   For  $\bu\in [H^1(T)]^d$, the following property holds:
\begin{equation}\label{pro1}
\nabla_{w}\bu ={\cal Q}_h(\nabla \bu), 
\end{equation}
\begin{equation}\label{pro2}
\nabla_{w} \cdot \bu ={\cal Q}_h(\nabla \cdot \bu),  
\end{equation}
\begin{equation}\label{pro3}
\nabla_w\cdot Q_h\bu={\cal Q}_h(\nabla \cdot \bu).
\end{equation}
\end{lemma}

\begin{proof} For any $\bu\in [H^1(T)]^d$, using \eqref{2.4new}, we have  
 \begin{equation*} 
  \begin{split}
 &(\nabla_w \bu, \bv)_T\\
  =&(\nabla \bu,  \bv)_T+
  \langle \bu|_{\partial T}-\bu|_T, \bv \cdot\bn \rangle_{\partial T} \\
  =&(\nabla \bu,  \bv)_T =({\cal Q}_h\nabla \bu,  \bv)_T ,
   \end{split}
   \end{equation*} 
  for all $\bv\in [P_r(T)]^{d\times d}$.  

For any $\bu\in [H^1(T)]^d$, using \eqref{divnew}, we have  
  \begin{equation*} 
  \begin{split}
 &(\nabla_w \cdot \bu, w)_T\\
  =&(\nabla \cdot \bu,  w)_T+
  \langle (\bu|_{\partial T}-\bu|_T) \cdot\bn, w  \rangle_{\partial T} \\
  =&(\nabla \cdot \bu,  w)_T =({\cal Q}_h\nabla \cdot \bu,  w)_T,
   \end{split}
   \end{equation*} 
  for all $w\in  P_r(T)$.  

Finally, for  any $\bu\in [H^1(T)]^d$, using \eqref{div}, we have  
  \begin{equation*} 
  \begin{split}
 &(\nabla_w \cdot Q_h\bu, w)_T\\
  =&-(Q_0\bu, \nabla w)_T+
  \langle  Q_b\bu \cdot\bn, w  \rangle_{\partial T} \\
  =&-( \bu, \nabla w)_T+
  \langle \bu \cdot\bn, w  \rangle_{\partial T} \\
  =&(\nabla \cdot \bu,  w)_T =({\cal Q}_h\nabla \cdot \bu,  w)_T,
   \end{split}
   \end{equation*} 
  for all $w\in  P_r(T)$.  
  
  This  completes the proof.

  \end{proof}

Let  $\bu$ and $p$ be the exact solutions of the Stokes equations \eqref{model}, and let  
  $\bu_h \in V_{h}^0$ and $p_h\in W_h$ denote their numerical approximations obtained from the WG scheme  \ref{PDWG1}. The error functions, denoted as  $e_{\bu_h}$ and $e_{p_h}$, are defined as 
\begin{equation}\label{error} 
e_{\bu_h}=\bu-\bu_h,  \qquad e_{p_h}=p-p_h.
\end{equation}

\begin{lemma}\label{errorequa}
The error functions $e_{\bu_h}$ and $e_{p_h}$ defined in  \eqref{error}  satisfy  the following error equations:
\begin{equation}\label{erroreqn}
\begin{split}
 (\nabla_{w} e_{\bu_h}, \nabla_w  \bv_h) -(\nabla_w \cdot \bv_h, e_{p_h})  =&\ell_1 (\bu, \bv_h)+\ell_2 (  \bv_h, p),  \qquad \forall \bv_h\in V_h^0, \\(\nabla_w \cdot e_{\bu_h}, q_h)=&0, \qquad \forall q_h\in W_h,
\end{split}
\end{equation}
where 
$$
\ell_1 (\bu, \bv_h)=\sum_{T\in {\cal T}_h}  
  \langle \bv_b-\bv_0,   ({\cal Q}_h-I) \nabla \bu \cdot \bn \rangle_{\partial T},
$$ 
$$
\ell_2 (  \bv_h, p)=\sum_{T\in {\cal T}_h} -\langle ({\cal Q}_h -I)p, (\bv_b-\bv_0)\cdot \bn \rangle_{\partial T}.
$$
\end{lemma}
\begin{proof} Using \eqref{pro1}, standard  integration by parts, and setting $\bvarphi= {\cal Q}_h \nabla \bu$ in  \eqref{2.4new}, we obtain 
\begin{equation}\label{term1}
\begin{split}
&\sum_{T\in {\cal T}_h} (\nabla_{w} \bu, \nabla_w  \bv_h)_T\\=&\sum_{T\in {\cal T}_h}\sum_{T\in {\cal T}_h} ({\cal Q}_h(\nabla \bu), \nabla_w  \bv_h)_T\\ 
=&\sum_{T\in {\cal T}_h}   (\nabla \bv_0,  {\cal Q}_h \nabla \bu)_T+
  \langle \bv_b-\bv_0,   {\cal Q}_h \nabla \bu \cdot \bn \rangle_{\partial T}\\
=& \sum_{T\in {\cal T}_h}(\nabla \bv_0,   \nabla \bu)_T+
  \langle \bv_b-\bv_0,   {\cal Q}_h \nabla \bu \cdot \bn \rangle_{\partial T}\\
  =& \sum_{T\in {\cal T}_h}-(  \bv_0,   \Delta \bu)_T+\langle \nabla \bu\cdot\bn, \bv_0\rangle_{\partial T}+
  \langle \bv_b-\bv_0,   {\cal Q}_h \nabla \bu \cdot \bn \rangle_{\partial T}\\
  =& \sum_{T\in {\cal T}_h}-(  \bv_0,   \Delta \bu)_T+ 
  \langle \bv_b-\bv_0,   ({\cal Q}_h-I) \nabla \bu \cdot \bn \rangle_{\partial T},
\end{split}
\end{equation}
where we used $\sum_{T\in {\cal T}_h} \langle \nabla \bu\cdot\bn, \bv_b\rangle_{\partial T}=\langle \nabla \bu\cdot\bn, \bv_b\rangle_{\partial \Omega}=0$ since  $\bv_b=0$ on $\partial \Omega$.

Using standard integration by parts and setting  $w={\cal Q}_h  p$ in  \eqref{divnew}, we obtain 
 \begin{equation}\label{termm}
     \begin{split}
  &\sum_{T\in {\cal T}_h} (\nabla_w \cdot \bv_h, p)_T\\=&
  \sum_{T\in {\cal T}_h} (\nabla_w \cdot \bv_h, {\cal Q}_h p)_T \\=& \sum_{T\in {\cal T}_h} (\nabla \cdot \bv_0,  {\cal Q}_h p)_T + \langle {\cal Q}_h p, (\bv_b-\bv_0)\cdot \bn \rangle_{\partial T}\\
  =& \sum_{T\in {\cal T}_h} (\nabla \cdot \bv_0,    p)_T + \langle {\cal Q}_h p, (\bv_b-\bv_0)\cdot \bn \rangle_{\partial T}\\
  =& \sum_{T\in {\cal T}_h}-( \bv_0,    \nabla p)_T 
+\langle p, \bv_0\cdot\bn\rangle_{\partial T}+ \langle {\cal Q}_h p, (\bv_b-\bv_0)\cdot \bn \rangle_{\partial T}\\    =& \sum_{T\in {\cal T}_h}-( \bv_0, \nabla p)_T + \langle ( {\cal Q}_h -I)p, (\bv_b-\bv_0)\cdot \bn \rangle_{\partial T},
\end{split}
 \end{equation}
where we used   $\sum_{T\in {\cal T}_h} \langle p, \bv_b\cdot\bn\rangle_{\partial T}=\langle p, \bv_b\cdot\bn\rangle_{\partial \Omega}=0$ since $\bv_b=0$ on $\partial \Omega$.

Subtracting \eqref{termm} from \eqref{term1}, and using the first equation in \eqref{model}, we obtain: 
\begin{equation*}  
\begin{split}
&\sum_{T\in {\cal T}_h}(\nabla_{w} \bu, \nabla_w  \bv_h)_T -(\nabla_w \cdot \bv_h, p)_T \\=& \sum_{T\in {\cal T}_h}-(  \bv_0,   \Delta \bu)_T+ 
  \langle \bv_b-\bv_0,   ({\cal Q}_h-I) \nabla \bu \cdot \bn \rangle_{\partial T} + 
 ( \bv_0, \nabla p)_T -\langle ({\cal Q}_h -I)p, (\bv_b-\bv_0)\cdot \bn \rangle_{\partial T}\\
 =&\sum_{T\in {\cal T}_h} (  \bv_0,   \bf)_T+ 
  \langle \bv_b-\bv_0,   ({\cal Q}_h-I) \nabla \bu \cdot \bn \rangle_{\partial T}   -\langle ({\cal Q}_h -I)p, (\bv_b-\bv_0)\cdot \bn \rangle_{\partial T}.
  \end{split}
\end{equation*}
Subtracting the first equation of  \eqref{WG}  from the above  yields  the first equation of \eqref{erroreqn}. 

Next, using \eqref{pro2} and the second equation of \eqref{model}, we find:
$$
0= (\nabla_w \cdot \bu, q_h)=\sum_{T\in {\cal T}_h}({\cal Q}_h \nabla\cdot \bu, q_h)_T=\sum_{T\in {\cal T}_h}(\nabla\cdot \bu, q_h)_T=0.
$$
Subtracting this result from the second equation of \eqref{WG} yields the second equation of \eqref{erroreqn}. 

This concludes the proof.

\end{proof}

\section{Error Estimates in $H^1$} This section establishes the error estimates in a discrete $H^1$ norm. 
\begin{lemma}\cite{wg21}\label{lem}
Let ${\cal T}_h$ be a finite element partition of the domain $\Omega$ satisfying the shape  regularity  assumption  specified in \cite{wy3655}. For any $0\leq s \leq 1$ and $1\leq m \leq k$, $1\leq  n \leq  2N+k-1$, the following estimates hold:
\begin{eqnarray}
\label{error1}
 \sum_{T\in {\cal T}_h} h_T^{2s}\|({\cal Q}_h-I)p\|^2_{s,T}&\leq& C  h^{2n}\|p\|^2_{n},\\
\label{error2}
\sum_{T\in {\cal T}_h}h_T^{2s}\|\bu- Q _0\bu\|^2_{s,T}&\leq& C h^{2(m+1)}\|\bu\|^2_{m+1},\\
\label{error3}\sum_{T\in {\cal T}_h}h_T^{2s}\|\nabla\bu-{\cal Q}_h(\nabla\bu)\|^2_{s,T}&\leq& C h^{2n}\|\bu\|^2_{n+1}.
\end{eqnarray}
 \end{lemma}
 \begin{lemma}\label{lemma1}
If   $\bu\in [H^{k+1}(\Omega)]^d$, then there exists a constant 
$C$
 such that 
\begin{equation}\label{erroresti1}
\3bar \bu-Q_h\bu \3bar \leq Ch^{k}\|\bu\|_{k+1}.
\end{equation}
\end{lemma}
\begin{proof}
Using \eqref{2.4new}, the trace inequalities \eqref{tracein} and \eqref{trace}, the Cauchy-Schwarz inequality, and the estimate \eqref{error2} for  $m=k$ and $s=0, 1$, we derive: for any $\bvarphi\in [P_r(T)]^{d\times d}$,
\begin{equation*}
\begin{split}
&|\sum_{T\in {\cal T}_h} (\nabla_w (\bu-Q_h\bu), \bvarphi)_T|\\=&  | \sum_{T\in {\cal T}_h} (\nabla (\bu-Q_0\bu), \bvarphi)_T-\langle Q_b\bu-Q_0\bu, \bvarphi \cdot\bn\rangle_{\partial T}|\\
\leq & (\sum_{T\in {\cal T}_h} \|\nabla (\bu-Q_0\bu)\|_T )^{\frac{1}{2}}(\sum_{T\in {\cal T}_h}\|\bvarphi\|_T^2)^{\frac{1}{2}}+(\sum_{T\in {\cal T}_h} \| Q_b\bu-Q_0\bu\|_{\partial T} ^2)^{\frac{1}{2}} (\sum_{T\in {\cal T}_h}\|\bvarphi \cdot\bn\|_{\partial T}^2)^{\frac{1}{2}}\\\leq & (\sum_{T\in {\cal T}_h} \|\nabla (\bu-Q_0\bu)\|_T )^{\frac{1}{2}}(\sum_{T\in {\cal T}_h}\|\bvarphi\|_T^2)^{\frac{1}{2}}
\\&+(\sum_{T\in {\cal T}_h} h_T^{-1}\| \bu-Q_0\bu\|_{ T} ^2+h_T \| \bu-Q_0\bu\|_{1,T} ^2)^{\frac{1}{2}} (\sum_{T\in {\cal T}_h}h_T^{-1}\|\bvarphi \|_{T}^2)^{\frac{1}{2}}\\
\leq & Ch^k\|\bu\|_{k+1} (\sum_{T\in {\cal T}_h} \|\bvarphi \|_{T}^2)^{\frac{1}{2}}.
\end{split}
\end{equation*}
Letting $\bvarphi=\nabla_w (\bu-Q_h\bu)$ yields 
$$
\sum_{T\in {\cal T}_h} (\nabla_w (\bu-Q_h\bu), \nabla_w (\bu-Q_h\bu))_T\leq 
 Ch^{k}\|\bu\|_{k+1}\3bar \bu-Q_h\bu \3bar.$$  
 
 This completes the proof.
\end{proof}

\begin{lemma}\label{lemma2}
If   $\bu\in [H^{k+1}(\Omega)]^d$, then there exists a constant 
$C$
 such that 
\begin{equation}\label{erroresti2}
(\sum_{T\in {\cal T}_h} \|\nabla_w\cdot (\bu-Q_h\bu)\|_T^2)^{
\frac{1}{2}
} \leq Ch^{k}\|\bu\|_{k+1}.
\end{equation}
\end{lemma}
\begin{proof}
    The proof follows the same approach as Lemma  \ref{lemma1}.
\end{proof}

\begin{lemma}
For any $\bu\in [H^{k+1}(\Omega)]^d$, $q\in H^{k}(\Omega)$, $\bv_h=\{\bv_0, \bv_b\}\in V_h^0$ and $q_h \in W_h$, the following estimates hold:
\begin{equation}\label{es1}
   |\ell_1(\bu, \bv_h)| \leq  Ch^k \|\bu\|_{k+1}\3bar \bv_h \3bar, 
\end{equation}
\begin{equation}\label{es2}
   |\ell_2(\bv_h, p)| \leq  Ch^k \|p\|_{k}\3bar \bv_h \3bar.
\end{equation}

\end{lemma}

\begin{proof}
Recall that ${\cal Q}_h$ denotes the $L^2$ projection operator onto the finite element space of piecewise polynomials of degree at most $2N+k-1\geq k$ for non-convex element and $N+k-1\geq k$ for convex element in the finite element partitions. 

Using  the Cauchy-Schwarz inequality, the trace inequality \eqref{tracein}, the norm equivalence \eqref{normeq}, and letting $n=k$ in \eqref{error3},  we have
 \begin{equation*} 
\begin{split}
|\ell_1(\bu, \bv_h)|\leq &(\sum_{T\in {\cal T}_h}  
  h_T^{-1}\|\bv_b-\bv_0\|_{\partial T} ^2)^{\frac{1}{2}} (\sum_{T\in {\cal T}_h}  
 h_T \|({\cal Q}_h-I) \nabla \bu \cdot \bn\|_{\partial T} ^2)^{\frac{1}{2}} \\
\leq & \|\bv_h\|_{1, h} (\sum_{T\in {\cal T}_h}  
 \|({\cal Q}_h-I) \nabla \bu \cdot \bn\|_{T} ^2+ h_T^2 \|({\cal Q}_h-I) \nabla \bu \cdot \bn\|_{1, T} ^2)^{\frac{1}{2}}\\ 
 \leq & Ch^k \|\bu\|_{k+1}\3bar \bv_h \3bar.
\end{split}
\end{equation*}
This completes the proof of \eqref{es1}.

Using the Cauchy-Schwarz inequality, the trace inequality \eqref{tracein}, letting $n=k$ in  \eqref{error1}, the norm equivalence \eqref{normeq},   we have
 \begin{equation*} 
\begin{split}
|\ell_2(\bv_h, p)|\leq &(\sum_{T\in {\cal T}_h}  
  h_T^{-1}\|\bv_b-\bv_0\|_{\partial T} ^2)^{\frac{1}{2}} (\sum_{T\in {\cal T}_h}  
 h_T \|({\cal Q}_h-I)p\|_{\partial T} ^2)^{\frac{1}{2}} \\
\leq & \|\bv_h\|_{1, h} (\sum_{T\in {\cal T}_h}  
 \|({\cal Q}_h-I)p\|_{T} ^2+ h_T^2 \|({\cal Q}_h-I)p\|_{1, T} ^2)^{\frac{1}{2}}\\ 
 \leq & Ch^k \|p\|_{k}\3bar \bv_h \3bar.
\end{split}
\end{equation*}
This completes the proof of \eqref{es2}.

\end{proof}
\begin{theorem}
Suppose  the exact solutions 
$\bu$ and $p$
 of the Stokes problem \eqref{model} satisfy 
$\bu\in [H^{k+1}(\Omega)]^d$ and $p\in H^k(\Omega)$. Then the error estimate satisfies:
\begin{equation}\label{trinorm}
\3bar \bu-\bu_h\3bar +\|p-p_h\|\leq Ch^{k}(\|\bu\|_{k+1}+\|p\|_k).
\end{equation}
\end{theorem}
\begin{proof}
 Letting $\bv_h=Q_h\bu-\bu_h$  in the first equation of \eqref{erroreqn} gives
  \begin{equation} \label{er}
\begin{split}
\3bar e_{\bu_h}\3bar ^2=& \sum_{T\in {\cal T}_h} (\nabla_w  e_{\bu_h} , \nabla_w   e_{\bu_h} )_T\\
=& \sum_{T\in {\cal T}_h} (\nabla_w  e_{\bu_h} , \nabla_w    (\bu-Q_h\bu) )_T+ (\nabla_w  e_{\bu_h} , \nabla_w    ( Q_h\bu-\bu_h) )_T\\
=& \sum_{T\in {\cal T}_h} (\nabla_w  e_{\bu_h} , \nabla_w    (\bu-Q_h\bu) )_T+ \ell_1(\bu, Q_h\bu-\bu_h) \\
&+\ell_2 (Q_h\bu-\bu_h, p)+\sum_{T\in {\cal T}_h} (\nabla_w \cdot(Q_h\bu-\bu_h), e_{p_h})_T\\
 =&I_1+I_2+I_3+I_4. 
\end{split}
\end{equation}

We will estimate $I_i$ for $i=1, \cdots, 4$.

For $I_1$, using the Cauchy-Schwarz inequality, and Lemma \ref{lemma1}, we have
  \begin{equation*} 
\begin{split}
\sum_{T\in {\cal T}_h} (\nabla_w  e_{\bu_h} , \nabla_w    (\bu-Q_h\bu) )_T\leq & \3bar   e_{\bu_h}\3bar \3bar       \bu-Q_h\bu  \3bar
\\ \leq & Ch^k\|\bu\|_{k+1}\3bar   e_{\bu_h}\3bar.
\end{split}
\end{equation*}

For $I_2$,  letting $\bv_h=Q_h\bu-\bu_h$ in \eqref{es1}, and using Lemma \ref{lemma1} and the triangle inequality, we have
\begin{equation*} 
\begin{split}
|\ell_1(\bu, Q_h\bu-\bu_h)|\leq & Ch^k\|\bu\|_{k+1}\3bar  Q_h\bu-\bu_h \3bar \\
\leq & Ch^k\|\bu\|_{k+1}(\3bar  Q_h\bu- \bu  \3bar+  \3bar   \bu-\bu_h \3bar)
\\  \leq & Ch^k\|\bu\|_{k+1}(h^k\|\bu\|_{k+1}+  \3bar   \bu-\bu_h \3bar).
\end{split}
\end{equation*}

For $I_3$,  letting $\bv_h=Q_h\bu-\bu_h$ in \eqref{es2}, and using Lemma \ref{lemma1}, the triangle inequality   and the Young's inequality, we obtain
\begin{equation*} 
\begin{split}
|\ell_2 (Q_h\bu-\bu_h, p)|\leq & Ch^k\|p\|_{k+1}\3bar  Q_h\bu-\bu_h \3bar \\
\leq & Ch^k\|p\|_{k+1}(\3bar  Q_h\bu- \bu  \3bar+  \3bar   \bu-\bu_h \3bar)
\\  \leq & Ch^k\|p\|_{k+1}(h^k\|\bu\|_{k+1}+  \3bar   \bu-\bu_h \3bar) 
 \\  \leq & C_1h^{2k}\|p\|^2_{k+1}+C_2 h^{2k}\|\bu\|^2_{k+1}+ Ch^k\|p\|_{k+1}  \3bar   \bu-\bu_h \3bar.
\end{split}
\end{equation*}

 Now we estimate $I_4$.   
Using \eqref{pro3} and the second equation in \eqref{model} gives
\begin{equation} \label{ss2}
\begin{split}
 \sum_{T\in {\cal T}_h}(\nabla_w\cdot Q_h\bu,  p-p_h)_T 
=&\sum_{T\in {\cal T}_h}({\cal Q}_h(\nabla \cdot  \bu),   p-p_h)_T=0.
\end{split}
\end{equation}
Using \eqref{ss2}, the estimate \eqref{error1} with $n=k$,  Lemma \ref{lemma2}, the fact that $(\nabla_w \cdot \bu_h,  p-{\cal Q}_hp)_T=0$, the Cauchy-Schwarz inequality, \eqref{pro3},  the second equation in \eqref{model},  and Young's inequality  that 
 \begin{equation*} 
\begin{split}
&|\sum_{T\in {\cal T}_h}(\nabla_w \cdot(Q_h\bu-\bu_h), e_{p_h})_T|\\=&  |\sum_{T\in {\cal T}_h}(\nabla_w \cdot \bu_h,  p-p_h)_T  |\\
= &  |\sum_{T\in {\cal T}_h}(\nabla_w \cdot \bu_h,  p-{\cal Q}_hp)_T+ (\nabla_w \cdot \bu_h,  {\cal Q}_hp-p_h)_T |\\
= &  |\sum_{T\in {\cal T}_h}  (\nabla_w \cdot \bu_h,  {\cal Q}_hp-p_h)_T |\\
= &  |\sum_{T\in {\cal T}_h}  (\nabla_w \cdot (\bu_h-Q_h\bu),  {\cal Q}_hp-p_h)_T+(\nabla_w \cdot  Q_h\bu,  {\cal Q}_hp-p_h)_T |\\
= &  |\sum_{T\in {\cal T}_h}  (\nabla_w \cdot (\bu_h-Q_h\bu),  {\cal Q}_hp-p_h)_T+ ({\cal Q}_h(\nabla  \cdot   \bu),  {\cal Q}_hp-p_h)_T |
\\=&  |\sum_{T\in {\cal T}_h} (\nabla_w \cdot (\bu_h-Q_h\bu),  {\cal Q}_hp-p_h)_T  |\\
\leq & (\sum_{T\in {\cal T}_h} \|(\nabla_w \cdot(Q_h\bu-\bu)\| _T^2)^{\frac{1}{2}}  (\sum_{T\in {\cal T}_h} \|{\cal Q}_hp-p_h\| _T^2)^{\frac{1}{2}}  \\
\leq & C  h^k\|\bu\|_{k+1}h^k\|p\|_k  \\
\leq & C_1h^{2k}\|p\|^2_{k+1}+C_2 h^{2k}\|\bu\|^2_{k+1}.
\end{split}
\end{equation*}
 
  Substituting the bounds for $I_i$ for $i=1,\cdots, 4$ into \eqref{er} gives
 \begin{equation*} 
\begin{split}
 \3bar e_{\bu_h}\3bar^2 \leq  &Ch^k\|\bu\|_{k+1} \3bar e_{\bu_h}\3bar+Ch^k\|\bu\|_{k+1}(h^k\|\bu\|_{k+1}+  \3bar   \bu-\bu_h \3bar)\\&+  C_1h^{2k}\|p\|^2_{k+1}+C_2 h^{2k}\|\bu\|^2_{k+1}+ Ch^k\|p\|_{k+1}  \3bar   \bu-\bu_h \3bar. 
 \end{split}
\end{equation*}

   This gives
\begin{equation}\label{euh}
   \3bar e_{\bu_h}\3bar\leq Ch^k(\|\bu\|_{k+1}+ \|p\|_{k}). 
\end{equation}

Finally, using the first equation of \eqref{erroreqn}, the fact that $(\nabla_w \cdot\bv_h, {\cal Q}_h p-p )_T=0$, \eqref{es1}, \eqref{es2}, and the Cauchy-Schwarz inequality gives 
 \begin{equation*} 
\begin{split}
&|\sum_{T\in {\cal T}_h}(\nabla_w \cdot\bv_h, {\cal Q}_h p-p_h)_T|\\
\leq&|\sum_{T\in {\cal T}_h}(\nabla_w \cdot\bv_h, {\cal Q}_h p-p )_T+(\nabla_w \cdot\bv_h, p-p_h)_T|\\
\leq&
 |\ell_1(\bu, \bv_h)|+|\ell_2(  \bv_h, p)|+|(\nabla_w  e_{\bu_h}, \nabla_w\bv_h)|\\
\leq &  Ch^k\|\bu\|_{k+1}\3bar  \bv_h\3bar+ Ch^k \|p\|_{k}\3bar \bv_h  \3bar +\3bar e_{\bu_h} \3bar \3bar \bv_h\3bar.
\end{split}
\end{equation*}
Using  the inf-sup condition \eqref{infsup}, and  \eqref{euh} leads to 
\begin{equation*} 
\begin{split}
\|{\cal Q}_h p-p_h \|\leq &C\frac{|\sum_{T\in {\cal T}_h}(\nabla_w \cdot\bv_h, {\cal Q}_h p-p_h)_T|}{\3bar \bv_h\3bar}\\
\leq & Ch^k\|\bu\|_{k+1}  + Ch^k \|p\|_{k}  +\3bar e_{\bu_h} \3bar 
\\
\leq & Ch^k(\|\bu\|_{k+1}  +   \|p\|_{k}). \end{split}
\end{equation*}
This, together with \eqref{error1}  with $n=k$ and the triangle inequality, yields 
$$
\|e_{p_h}\|\leq \|{\cal Q}_h p-p_h\|+\|p-{\cal Q}_h p\|\leq Ch^k(\|\bu\|_{k+1}  +   \|p\|_{k}).
$$
This completes the proof.
\end{proof}

\section{  $L^2$ Error Estimates}
To derive the error estimate in the $L^2$ norm, we employ the standard duality argument. Recall that the error function is  expressed as $e_{\bu_h}=\bu-\bu_h=\{e_{\bu_0}, e_{\bu_b}\}=\{\bu-\bu_0, \bu-\bu_b\}\in V_h^0$. Define $\bzeta_h =Q_h\bu - \bu_h=\{\bzeta_0, \bzeta_b\}=\{Q_0\bu - \bu_0, Q_b\bu - \bu_b\}\in V_h^0$. Consider the dual problem associated with the Stokes problem \eqref{model}, which seeks a function $\bw \in [H^2(\Omega)]^d$ and $q\in H^1(\Omega)$ satisfying:
\begin{equation}\label{dual}
\begin{split}
   - \Delta \bw+\nabla q &=\bzeta_0, \qquad \text{in}\ \Omega,\\
\nabla \cdot\bw&=0,\qquad \text{in}\ \Omega,  \\
\bw&=0, \qquad \text{on}\ \partial\Omega.
    \end{split}
\end{equation}
We assume the following regularity condition for the dual problem holds:
\begin{equation}\label{regu2}
 \|\bw\|_2+\|q\|_1\leq C\|\bzeta_0\|.
 \end{equation}
 
 \begin{theorem}
Let $\bu\in [H^{k+1}(\Omega)]^d$ and $p\in H^k(\Omega)$ be the exact solutions of the Stokes problem  \eqref{model}, and let 
 $\bu_h\in V_h^0$ and $p_h\in W_h$ denote the numerical solutions  obtained  using  the weak Galerkin scheme \ref{PDWG1}. Assume that the regularity condition   \eqref{regu2}   holds. Then, there exists a constant $C$ such that 
\begin{equation*}
\|\bu-\bu_0\|\leq Ch^{k+1}(\|\bu\|_{k+1}+\|p\|_k).
\end{equation*}
 \end{theorem}
 
 \begin{proof}
 Testing the first equation in the dual problem   \eqref{dual} with  $\bzeta_0$ gives:
 \begin{equation}\label{e1}
 \begin{split}
 \|\bzeta_0\|^2 =&(- \Delta \bw+\nabla q, \bzeta_0).
 \end{split}
 \end{equation}

 Using \eqref{term1} with  $\bu=\bw$, $\bv_h=\bzeta_h$, we have:
$$
\sum_{T\in {\cal T}_h}(- \Delta \bw, \bzeta_0)_T=\sum_{T\in {\cal T}_h} (\nabla_w\bw, \nabla_w \bzeta_h)_T-\langle \bzeta_b-\bzeta_0, ({\cal Q}_h-I)\nabla \bw\cdot\bn \rangle_{\partial T}.
$$

Similarly, using  \eqref{termm} with $p=q$ and $\bv_h=\bzeta_h$,  we obtain:
$$
\sum_{T\in {\cal T}_h} (\nabla q, \zeta_0)=\sum_{T\in {\cal T}_h} -(\nabla_w \cdot \bzeta_h, {\cal Q}_hq)_T+\langle ({\cal Q}_h-I)q, (\bzeta_b-\bzeta_0)\cdot\bn\rangle_{\partial T}.
$$
Substituting these into \eqref{e1} gives:
\begin{equation}\label{e2}
 \begin{split}
 \|\bzeta_0\|^2 =&\sum_{T\in {\cal T}_h} (\nabla_w\bw, \nabla_w \bzeta_h)_T-\langle \bzeta_b-\bzeta_0, ({\cal Q}_h-I)\nabla \bw\cdot\bn \rangle_{\partial T}\\&-(\nabla_w \cdot \bzeta_h, {\cal Q}_hq)_T+\langle ({\cal Q}_h-I)q, (\bzeta_b-\bzeta_0)\cdot\bn\rangle_{\partial T}.
 \end{split}
 \end{equation}

Using the second inequality in \eqref{dual} and \eqref{pro3}, we find: 
  \begin{equation}\label{add}
 \begin{split}
  \sum_{T\in {\cal T}_h}(\nabla_w\cdot Q_h\bw, p-p_h)_T=  \sum_{T\in {\cal T}_h} ({\cal Q}_h(\nabla\cdot \bw), p-p_h)_T=0.
  \end{split}
 \end{equation}

  Thus, using \eqref{add}  and  \eqref{erroreqn} with $\bv_h=Q_h\bw$ and $q_h={\cal Q}_hq$ gives

  \begin{equation}\label{e3}
 \begin{split}
 &\|\bzeta_0\|^2 \\=&\sum_{T\in {\cal T}_h} (\nabla_w\bw, \nabla_w (\bu-\bu_h))_T -(\nabla_w\bw, \nabla_w (\bu-Q_h\bu))_T\\&-\langle \bzeta_b-\bzeta_0, ({\cal Q}_h-I)\nabla \bw\cdot\bn \rangle_{\partial T} -(\nabla_w \cdot (\bu-\bu_h), {\cal Q}_hq)_T\\&-(\nabla_w \cdot (Q_h\bu-\bu), {\cal Q}_hq)_T+\langle ({\cal Q}_h-I)q, (\bzeta_b-\bzeta_0)\cdot\bn\rangle_{\partial T}\\
 =&\sum_{T\in {\cal T}_h} (\nabla_wQ_h\bw, \nabla_w (\bu-\bu_h))_T+(\nabla_w (\bw-Q_h\bw), \nabla_w (\bu-\bu_h))_T\\
 &- (\nabla_w  \bw, \nabla_w (\bu-Q_h\bu))_T-\langle \bzeta_b-\bzeta_0, ({\cal Q}_h-I)\nabla \bw\cdot\bn \rangle_{\partial T}\\&-(\nabla_w \cdot (\bu-\bu_h), {\cal Q}_hq)_T-(\nabla_w \cdot (Q_h\bu-\bu), {\cal Q}_hq)_T\\
 &+\langle ({\cal Q}_h-I)q, (\bzeta_b-\bzeta_0)\cdot\bn\rangle_{\partial T} -(\nabla_w\cdot Q_h\bw, p-p_h)_T\\
 =& \ell_1(\bu, Q_h\bw) +\ell_2(Q_h\bw, p) +\sum_{T\in {\cal T}_h}(\nabla_w (\bw-Q_h\bw), \nabla_w (\bu-\bu_h))_T\\
 &- (\nabla_w  \bw, \nabla_w (\bu-Q_h\bu))_T-\langle \bzeta_b-\bzeta_0, ({\cal Q}_h-I)\nabla \bw\cdot\bn \rangle_{\partial T}\\&-(\nabla_w \cdot (\bu-\bu_h),  q)_T-(\nabla_w \cdot (Q_h\bu-\bu),  q)_T\\
 &+\langle ({\cal Q}_h-I)q, (\bzeta_b-\bzeta_0)\cdot\bn\rangle_{\partial T}\\
 =& \sum_{i=1}^8 I_i.
 \end{split}
 \end{equation}

Each term $I_i$ for $i=1, \cdots, 8$ is estimated as follows:
 
For $I_1$, using the Cauchy-Schwarz inequality, the trace inequality \eqref{tracein},   the  estimate  \eqref{error2} with $m=1$,  the  estimate  \eqref{error3} with $n=k$, we have
\begin{equation*} 
 \begin{split}
&|\ell_1(\bu, Q_h\bw)|\\=&|\sum_{T\in {\cal T}_h}  
  \langle Q_b\bw-Q_0\bw,   ({\cal Q}_h-I) \nabla \bu \cdot \bn \rangle_{\partial T}|\\
  \leq & (\sum_{T\in {\cal T}_h}  
  \| Q_b\bw-Q_0\bw\|^2_{\partial T})^
{\frac{1}{2}} (\sum_{T\in {\cal T}_h}\| ({\cal Q}_h-I) \nabla \bu \cdot \bn \|^2_{\partial T})
\\
\leq & (\sum_{T\in {\cal T}_h}  
  h_T^{-1}\|  \bw-Q_0\bw\|^2_{T}+h_T\|  \bw-Q_0\bw\|^2_{1, T})^
{\frac{1}{2}}\\
&\cdot(\sum_{T\in {\cal T}_h}  h_T^{-1}\| ({\cal Q}_h-I) \nabla \bu \cdot \bn \|^2_{T}+h_T \| ({\cal Q}_h-I) \nabla \bu \cdot \bn \|^2_{1, T})\\
\leq & Ch^{-1}h^2\|\bw\|_2 h^k\|\bu\|_{k+1}.
  \end{split}
 \end{equation*}

For $I_2$, applying the Cauchy-Schwarz inequality, the trace inequality \eqref{tracein},   the   estimate  \eqref{error1} with $n=k$,  the   estimate  \eqref{error2} with $m=1$, we get
\begin{equation*} 
 \begin{split}
&|\ell_2(Q_h\bw, p)| \\=&|\sum_{T\in {\cal T}_h} -\langle ({\cal Q}_h -I)p, (Q_b\bw-Q_0\bw)\cdot \bn \rangle_{\partial T}|\\
\leq & (\sum_{T\in {\cal T}_h}\| ({\cal Q}_h -I)p\|_{\partial T}^2)^\frac{1}{2}(\sum_{T\in {\cal T}_h} \|Q_b\bw-Q_0\bw)\cdot \bn \|_{\partial T}^2)^\frac{1}{2}\\
\leq & (\sum_{T\in {\cal T}_h}h_T^{-1}\| ({\cal Q}_h -I)p\|_{T}^2+h_T\| ({\cal Q}_h -I)p\|_{1, T}^2)^\frac{1}{2}\\&\cdot (\sum_{T\in {\cal T}_h} h_T^{-1}\|(\bw-Q_0\bw)\cdot \bn \|_{T}^2+h_T \|(\bw-Q_0\bw)\cdot \bn \|_{1, T}^2)^\frac{1}{2}\\
\leq &Ch^{-1}h^2\|\bw\|_2 h^k\|p\|_{k}.
 \end{split}
 \end{equation*} 

For $I_3$, using the Cauchy-Schwarz inequality,  the error estimate \eqref{trinorm}, the error estimate   \eqref{erroresti1} with $k=1$,  we derive
\begin{equation*} 
 \begin{split}
&|\sum_{T\in {\cal T}_h}(\nabla_w (\bw-Q_h\bw), \nabla_w (\bu-\bu_h))_T|\\
\leq & \3bar \bw-Q_h\bw\3bar \3bar \bu-\bu_h\3bar\\ \leq & Ch\|\bw\|_{2}  h^k(\|\bu\|_{k+1}+\|p\|_k).
 \end{split}
 \end{equation*}

We now estimate $I_4$. Let 
 $Q_1$ denote  the $L^2$ projection onto $[P_1(T)]^{d\times d}$. For any $T\in {\cal T}_h$, it follows from \eqref{2.4} that
$$
(Q_1 (\nabla_w  \bw), \nabla_w (\bu-Q_h\bu))_T= -( \bu-Q_0\bu, \nabla\cdot(Q_1(\nabla_w  \bw)))_T+\langle \bu-Q_b\bu, Q_1(\nabla_w  \bw) \cdot\bn\rangle_{\partial T}=0.
$$
Using this identity along with the Cauchy-Schwarz inequality,   \eqref{pro1}, and \eqref{erroresti1}, we have 
\begin{equation*}
\begin{split}
&|\sum_{T\in {\cal T}_h}(\nabla_w  \bw, \nabla_w (\bu-Q_h\bu))_T| \\
=&|\sum_{T\in {\cal T}_h}(\nabla_w  \bw-Q_1 (\nabla_w  \bw), \nabla_w (\bu-Q_h\bu))_T| \\
=&|\sum_{T\in {\cal T}_h}({\cal Q}_h(\nabla   \bw)-Q_1 ({\cal Q}_h(\nabla   \bw)), \nabla_w (\bu-Q_h\bu))_T|  \\
\leq & (\sum_{T\in {\cal T}_h}\|{\cal Q}_h(\nabla   \bw)-Q_1 ({\cal Q}_h(\nabla   \bw))\|_T^2)^\frac{1}{2}(\sum_{T\in {\cal T}_h}\| \nabla_w (\bu-Q_h\bu)\|_T^2)^\frac{1}{2}\\
\leq & Ch\|{\cal Q}_h(\nabla   \bw)\|_1 h^k\|\bu\|_{k+1}
\\ \leq & Ch^{k+1}\|\bw\|_2  \|\bu\|_{k+1}.
 \end{split}
 \end{equation*}

  For $I_5$, applying the Cauchy-Schwarz inequality, the trace inequality \eqref{tracein}, the norm equivalence \eqref{normeq}, \eqref{error3} with $n=1$, the triangle inequality,  and the error estimates  \eqref{erroresti1} and \eqref{trinorm}, we have
\begin{equation*}
\begin{split}
&|\sum_{T\in {\cal T}_h}\langle \bzeta_b-\bzeta_0, ({\cal Q}_h-I)\nabla \bw\cdot\bn \rangle_{\partial T}|\\
\leq & (\sum_{T\in {\cal T}_h}h_T^{-1}\|\bzeta_b-\bzeta_0\|_{\partial T}^2)^\frac{1}{2}  (\sum_{T\in {\cal T}_h} h_T\|({\cal Q}_h-I)\nabla \bw\cdot\bn \|_{\partial T}^2)^\frac{1}{2}\\
\leq & \|\bzeta_h\|_{1, h} (\sum_{T\in {\cal T}_h} \|({\cal Q}_h-I)\nabla \bw\cdot\bn \|_{T}^2+h_T^2\|({\cal Q}_h-I)\nabla \bw\cdot\bn \|_{1,T}^2)^\frac{1}{2}\\
 \leq & \3bar \bzeta_h\3bar  h\|\bw\|_2
 \\
  \leq & (\3bar Q_h\bu-\bu\3bar+\3bar \bu-\bu_h\3bar ) h\|\bw\|_2\\
 \leq & Ch^{k+1}(\|\bu\|_{k+1}+\|p\|_k)\|\bw\|_2.
 \end{split}
 \end{equation*} 

Now we estimate  $I_6$. First, using the second equation in \eqref{model},  \eqref{pro2}-\eqref{pro3}, we have 
\begin{equation}\label{qq}
    \nabla_w \cdot \bu={\cal Q}_h \nabla  \cdot \bu=0, \qquad \nabla_w \cdot Q_h\bu={\cal Q}_h \nabla  \cdot \bu=0.
\end{equation}
Let  ${\cal Q}_h^{k-1}$ denote  the $L^2$ projection onto $P_{k-1}(T)$.  
Using the second equation in  \eqref{erroreqn}, the Cauchy-Schwarz inequality, the estimate \eqref{error1} with $n=1$, \eqref{qq},  the error estimates \eqref{erroresti2},
\begin{equation*}
\begin{split}
& |\sum_{T\in {\cal T}_h}|(\nabla_w \cdot (\bu-\bu_h),  q)_T|\\
=&|\sum_{T\in {\cal T}_h}|(\nabla_w \cdot (\bu-\bu_h),  q-{\cal Q}_h^{k-1}q)_T|\\
=&|\sum_{T\in {\cal T}_h}| ( \nabla_w \cdot (Q_h\bu-\bu_h),  q-{\cal Q}_h^{k-1}q )_T|\\
\leq & (\sum_{T\in {\cal T}_h}\|\nabla_w \cdot (Q_h\bu-\bu_h)\|_T^2)^{\frac{1}{2}}(\sum_{T\in {\cal T}_h}\|q-{\cal Q}_h^{k-1}q\|_T^2)^{\frac{1}{2}}\\
\leq & Ch^{k}\|\bu\|_{k+1}h \|q\|_1.
\end{split}
 \end{equation*} 

For $I_7$, let  $Q^0$ denote  the $L^2$ projection onto $P_0(T)$. For any $T\in {\cal T}_h$, we have from \eqref{div},
$$
(Q^0 q, \nabla_w \cdot ( Q_h\bu-\bu ))_T= -( Q_0\bu- \bu, \nabla (Q^0 q))_T+\langle  Q_b\bu-\bu, Q^0 q\cdot\bn\rangle_{\partial T}=0,
$$
which, combined with the  Cauchy-Schwarz inequality and the error estimate \eqref{erroresti2}, yields
\begin{equation*}
\begin{split}
& |\sum_{T\in {\cal T}_h}(\nabla_w \cdot (Q_h\bu-\bu),  q)_T|\\
\leq &  |\sum_{T\in {\cal T}_h}(\nabla_w \cdot (Q_h\bu-\bu),  q-Q^0q)_T|
\\
\leq & (\sum_{T\in {\cal T}_h}\|\nabla_w \cdot (Q_h\bu-\bu)\|_T^2)^\frac{1}{2}  (\sum_{T\in {\cal T}_h}\| q-Q^0q\|_T^2)^\frac{1}{2}\\
\leq & Ch^k\|\bu\|_{k+1 } h\|q\|_1.
 \end{split}
 \end{equation*}

Finally, for $I_8$, using the Cauchy-Schwarz inequality,  the trace inequality \eqref{tracein}, the norm equivalence \eqref{normeq}, the estimate \eqref{error1} with $n=1$, the triangle inequality, and  the error estimates  \eqref{erroresti1} and  \eqref{trinorm},  we get
\begin{equation*}
\begin{split}
&|\sum_{T\in {\cal T}_h}\langle ({\cal Q}_h-I)q, (\bzeta_b-\bzeta_0)\cdot\bn\rangle_{\partial T}|\\
 \leq&  (\sum_{T\in {\cal T}_h}h_T\|({\cal Q}_h-I)q\|_{\partial T}  ^2)^\frac{1}{2}  (\sum_{T\in {\cal T}_h}h_T^{-1}\|(\bzeta_b-\bzeta_0)\cdot\bn\|_{\partial T}^2)^\frac{1}{2} \\
  \leq&  (\sum_{T\in {\cal T}_h} \|({\cal Q}_h-I)q\|_{T}  ^2+h_T^2\|({\cal Q}_h-I)q\|_{1, T}  ^2)^\frac{1}{2}  \|\bzeta_h\|_{1, h}\\
  \leq &  Ch\|q\|_1 (\3bar Q_h\bu-\bu\3bar+\3bar  \bu-\bu_h\3bar)\\
   \leq &  Ch\|q\|_1 h^k(\|\bu\|_{k+1}+\|p\|_k). \\
   \end{split}
 \end{equation*} 

Substituting the estimates for $I_i$ ($i=1, \cdots, 8$) into \eqref{e3} and using the regularity assumption \eqref{regu2} 
gives
\begin{equation*}
\|\bzeta_0\|\leq Ch^{k+1}(\|\bu\|_{k+1}+\|p\|_k),
\end{equation*}
which, together with the triangle inequality, leads to
\begin{equation*}
\|\bu-\bu_0\|\leq \|\bu-Q_0\bu\|+\|\bzeta_0\|\leq Ch^{k+1}(\|\bu\|_{k+1}+\|p\|_k).
\end{equation*}

This completes the proof of the theorem. 
\end{proof}

\section{Numerical tests}
 
In the 2D computations,  we solve the stationary Stokes equations \eqref{model} in
  the unit square domain $\Omega=(0,1)\times(0,1)$ with the exact solution
\an{\label{s-1} \b u=\p{-\partial_y g(x,y) \\ \partial_x g(x,y)}, \quad 
      p=(y-\frac 12)^3, }
where $g(x,y)=2^4 (x-x^2)^2 (y-y^2)^2$.

\begin{figure}[H] \centering
  \begin{picture}(420,100)(0,20)
  \put(-10,-445){\includegraphics[width=410pt]{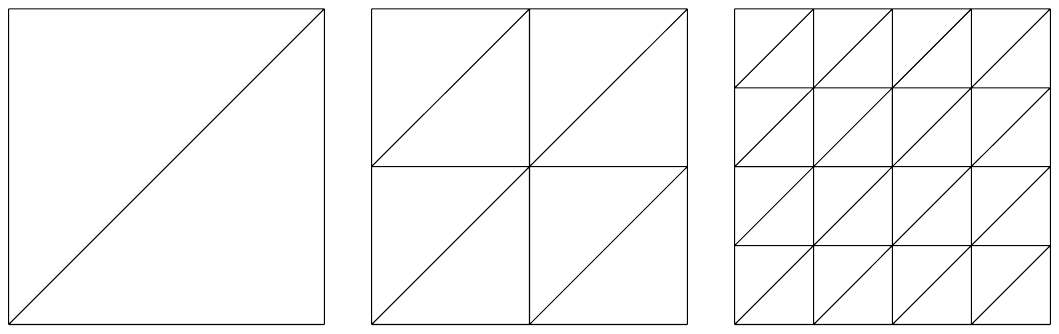}}
  \end{picture}
  \caption{The first three meshes $G_1$--$G_3$ for Table \ref{t2-1} } \label{g2-1}
	\end{figure}

We list the computational results for computing \eqref{s-1} on triangular meshes shown in
   Figure \ref{g2-1}, by the weak Galerkin finite element methods,  in Table \ref{t2-1}.

  \begin{table}[H]
  \caption{ Error profile for computing \eqref{s-1} on Figure \ref{g2-1} meshes.} \label{t2-1}
\begin{center}  
   \begin{tabular}{c|rr|rr|rr}  
 \hline 
$G_i$ & \multicolumn{2}{c|}{$\|\b u-\b u_h\|_0$ $O(h^r)$} & 
        \multicolumn{2}{c|}{$\|\nabla_w(\b u-\b u_h)\|_0$ $O(h^r)$}  & 
   \multicolumn{2}{c}{$\|Q_0p-p_h\|_0$  $O(h^r)$ } \\ \hline  
     & \multicolumn{6}{c}{By the $(P_1,P_1)$-$P_0$ weak Galerkin finite element}  \\ 
 6&     0.819E-4 &  2.0&     0.926E-2 &  1.0&     0.431E-2 &  1.0 \\
 7&     0.206E-4 &  2.0&     0.463E-2 &  1.0&     0.216E-2 &  1.0 \\
 8&     0.516E-5 &  2.0&     0.231E-2 &  1.0&     0.108E-2 &  1.0 \\
 \hline 
     & \multicolumn{6}{c}{By the $(P_2,P_2)$-$P_1$ weak Galerkin finite element}  \\ 
 4&     0.419E-4 &  3.2&     0.422E-2 &  1.9&     0.206E-2 &  1.9 \\
 5&     0.477E-5 &  3.1&     0.106E-2 &  2.0&     0.521E-3 &  2.0 \\
 6&     0.574E-6 &  3.1&     0.267E-3 &  2.0&     0.129E-3 &  2.0 \\
 \hline 
\end{tabular} \end{center}  \end{table}

\begin{figure}[H] \centering
  \begin{picture}(420,100)(0,20)
  \put(-10,-445){\includegraphics[width=410pt]{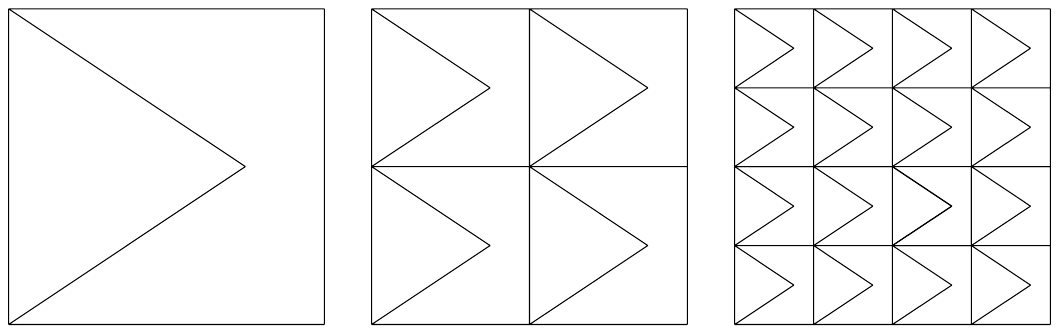}}
  \end{picture}
  \caption{The first three meshes $G_1$--$G_3$ for Table \ref{t2-2} } \label{g2-2}
	\end{figure}

We list the computational results for computing \eqref{s-1} on nonconvex polygonal meshes shown in
   Figure \ref{g2-2}, by the weak Galerkin finite element methods,  in Table \ref{t2-2}.

  \begin{table}[H]
  \caption{ Error profile for computing \eqref{s-1} on Figure \ref{g2-2} meshes.} \label{t2-2}
\begin{center}  
   \begin{tabular}{c|rr|rr|rr}  
 \hline 
$G_i$ & \multicolumn{2}{c|}{$\|\b u-\b u_h\|_0$ $O(h^r)$} & 
        \multicolumn{2}{c|}{$\|\nabla_w(\b u-\b u_h)\|_0$ $O(h^r)$}  & 
   \multicolumn{2}{c}{$\|Q_0p-p_h\|_0$  $O(h^r)$ } \\ \hline  
     & \multicolumn{6}{c}{By the $(P_1,P_1)$-$P_0$ weak Galerkin finite element}  \\ 
 6&     0.256E-3 &  1.9&     0.187E-1 &  1.0&     0.129E-2 &  1.7 \\
 7&     0.648E-4 &  2.0&     0.932E-2 &  1.0&     0.393E-3 &  1.7 \\
 8&     0.163E-4 &  2.0&     0.466E-2 &  1.0&     0.140E-3 &  1.5 \\
 \hline 
     & \multicolumn{6}{c}{By the $(P_2,P_2)$-$P_1$ weak Galerkin finite element}  \\ 
 4&     0.822E-4 &  3.4&     0.109E-1 &  2.9&     0.272E-2 &  2.0 \\
 5&     0.857E-5 &  3.3&     0.257E-2 &  2.1&     0.608E-3 &  2.2 \\
 6&     0.975E-6 &  3.1&     0.646E-3 &  2.0&     0.145E-3 &  2.1 \\
 \hline 
     & \multicolumn{6}{c}{By the $(P_3,P_3)$-$P_2$ weak Galerkin finite element}  \\ 
 2&     0.924E-2 &  6.0&     0.649E+0 &  4.4&     0.125E-1 &  1.3 \\
 3&     0.164E-3 &  5.8&     0.211E-1 &  4.9&     0.270E-2 &  2.2 \\
 4&     0.564E-5 &  4.9&     0.106E-2 &  4.3&     0.352E-3 &  2.9 \\
 \hline 
\end{tabular} \end{center}  \end{table}

\begin{figure}[H] \centering
  \begin{picture}(420,100)(0,20)
  \put(-10,-445){\includegraphics[width=410pt]{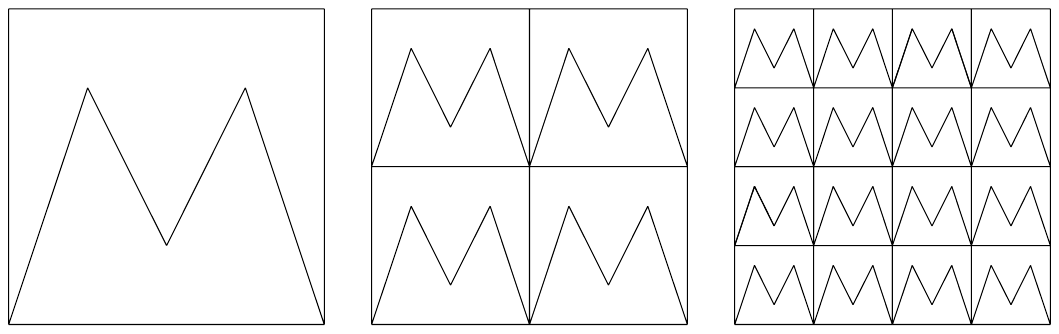}}
  \end{picture}
  \caption{The first three meshes $G_1$--$G_3$ for Table \ref{t2-3} } \label{g2-3}
	\end{figure}

We list the computational results for computing \eqref{s-1} on nonconvex polygonal meshes shown in
   Figure \ref{g2-3}, by the weak Galerkin finite element methods,  in Table \ref{t2-3}.

  \begin{table}[H]
  \caption{ Error profile for computing \eqref{s-1} on Figure \ref{g2-3} meshes.} \label{t2-3}
\begin{center}  
   \begin{tabular}{c|rr|rr|rr}  
 \hline 
$G_i$ & \multicolumn{2}{c|}{$\|\b u-\b u_h\|_0$ $O(h^r)$} & 
        \multicolumn{2}{c|}{$\|\nabla_w(\b u-\b u_h)\|_0$ $O(h^r)$}  & 
   \multicolumn{2}{c}{$\|Q_0p-p_h\|_0$  $O(h^r)$ } \\ \hline  
     & \multicolumn{6}{c}{By the $(P_1,P_1)$-$P_0$ weak Galerkin finite element}  \\ 
 5&     0.448E-3 &  1.9&     0.440E-1 &  1.0&     0.253E-2 &  1.5 \\
 6&     0.114E-3 &  2.0&     0.221E-1 &  1.0&     0.988E-3 &  1.4 \\
 7&     0.285E-4 &  2.0&     0.110E-1 &  1.0&     0.443E-3 &  1.2 \\
 \hline 
     & \multicolumn{6}{c}{By the $(P_2,P_2)$-$P_1$ weak Galerkin finite element}  \\ 
 4&     0.766E-4 &  3.0&     0.116E-1 &  2.0&     0.138E-2 &  2.3 \\
 5&     0.909E-5 &  3.1&     0.294E-2 &  2.0&     0.259E-3 &  2.4 \\
 6&     0.110E-5 &  3.0&     0.739E-3 &  2.0&     0.550E-4 &  2.2 \\
 \hline 
\end{tabular} \end{center}  \end{table}

\begin{figure}[H] \centering
  \begin{picture}(420,100)(0,20)
  \put(-10,-445){\includegraphics[width=410pt]{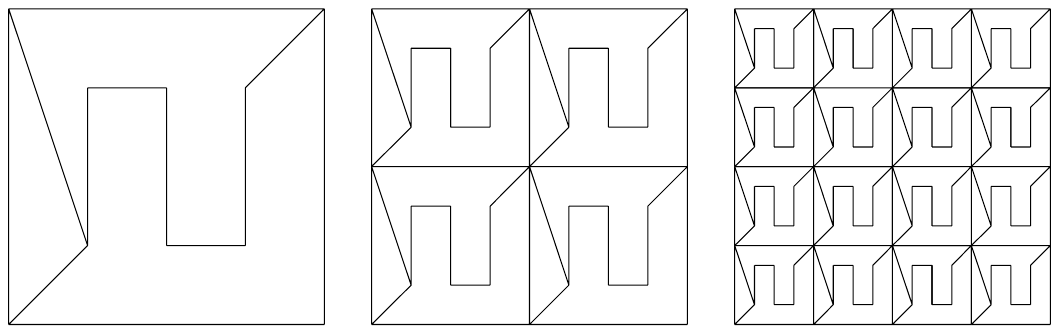}}
  \end{picture}
  \caption{The first three meshes $G_1$--$G_3$ for Table \ref{t2-4} } \label{g2-4}
	\end{figure}

We list the computational results for computing \eqref{s-1} on very nonconvex polygonal meshes shown in
   Figure \ref{g2-4}, by the weak Galerkin finite element methods,  in Table \ref{t2-4}.

  \begin{table}[H]
  \caption{ Error profile for computing \eqref{s-1} on Figure \ref{g2-4} meshes.} \label{t2-4}
\begin{center}  
   \begin{tabular}{c|rr|rr|rr}  
 \hline 
$G_i$ & \multicolumn{2}{c|}{$\|\b u-\b u_h\|_0$ $O(h^r)$} & 
        \multicolumn{2}{c|}{$\|\nabla_w(\b u-\b u_h)\|_0$ $O(h^r)$}  & 
   \multicolumn{2}{c}{$\|Q_0p-p_h\|_0$  $O(h^r)$ } \\ \hline  
     & \multicolumn{6}{c}{By the $(P_1,P_1)$-$P_0$ weak Galerkin finite element}  \\ 
 5&     0.365E-3 &  1.9&     0.508E-1 &  1.0&     0.311E-2 &  1.4 \\
 6&     0.927E-4 &  2.0&     0.255E-1 &  1.0&     0.135E-2 &  1.2 \\
 7&     0.233E-4 &  2.0&     0.128E-1 &  1.0&     0.642E-3 &  1.1 \\
 \hline 
     & \multicolumn{6}{c}{By the $(P_2,P_2)$-$P_1$ weak Galerkin finite element}  \\ 
 3&     0.130E-2 &  5.0&     0.132E+0 &  3.9&     0.577E-2 &  3.2 \\
 4&     0.898E-4 &  3.9&     0.131E-1 &  3.3&     0.124E-2 &  2.2 \\
 5&     0.111E-4 &  3.0&     0.271E-2 &  2.3&     0.272E-3 &  2.2 \\
 \hline 
\end{tabular} \end{center}  \end{table}

In the 3D computations,  we solve the stationary Stokes equations \eqref{model} on the
  domain $\Omega=(0,1)^3$.  The exact solution is chosen as
\an{\label{s-2} \b u=\p{e^y \\ e^z \\ e^x}, \quad \ p =y-\frac 12. }

\begin{figure}[H] \centering
  \begin{picture}(420,100)(0,20)
  \put(-10,-445){\includegraphics[width=410pt]{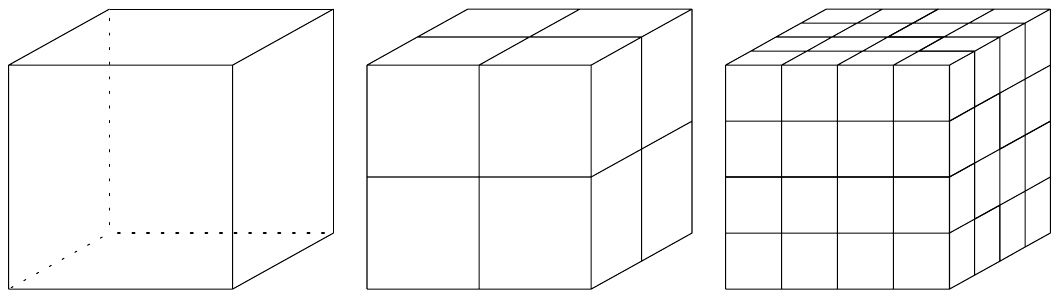}}
  \end{picture}
  \caption{The first three meshes $G_1$--$G_3$ for Table \ref{t3-1} } \label{g3-1}
	\end{figure}

We list the computational results for computing \eqref{s-2} on cubic meshes shown in
   Figure \ref{g3-1}, by the weak Galerkin finite element methods,  in Table \ref{t3-1}.

  \begin{table}[H]
  \caption{ Error profile for computing \eqref{s-2} on Figure \ref{g3-1} meshes.} \label{t3-1}
\begin{center}  
   \begin{tabular}{c|rr|rr|rr}  
 \hline 
$G_i$ & \multicolumn{2}{c|}{$\|\b u-\b u_h\|_0$ $O(h^r)$} & 
        \multicolumn{2}{c|}{$\|\nabla_w(\b u-\b u_h)\|_0$ $O(h^r)$}  & 
   \multicolumn{2}{c}{$\|Q_0p-p_h\|_0$  $O(h^r)$ } \\ \hline  
     & \multicolumn{6}{c}{By the $(P_1,P_1)$-$P_0$ weak Galerkin finite element}  \\ 
 3 &    0.494E-2 &1.82 &    0.277E+0 &0.89 &    0.108E+0 &1.36 \\
 4 &    0.122E-2 &2.02 &    0.142E+0 &0.96 &    0.321E-1 &1.75 \\
 5 &    0.301E-3 &2.02 &    0.718E-1 &0.99 &    0.899E-2 &1.83 \\
 \hline 
     & \multicolumn{6}{c}{By the $(P_2,P_2)$-$P_1$ weak Galerkin finite element}  \\ 
 2 &    0.109E-2 &2.48 &    0.284E-1 &1.75 &    0.779E-2 &0.00 \\
 3 &    0.144E-3 &2.92 &    0.745E-2 &1.93 &    0.981E-3 &2.99 \\
 4 &    0.183E-4 &2.97 &    0.190E-2 &1.97 &    0.139E-3 &2.82 \\
 \hline 
     & \multicolumn{6}{c}{By the $(P_3,P_3)$-$P_2$ weak Galerkin finite element}  \\ 
 1 &    0.384E-3 &0.00 &    0.804E-2 &0.00 &    0.396E-2 &0.00 \\
 2 &    0.260E-4 &3.88 &    0.114E-2 &2.82 &    0.168E-3 &4.56 \\
 3 &    0.175E-5 &3.89 &    0.147E-3 &2.96 &    0.114E-4 &3.88 \\
 \hline 
\end{tabular} \end{center}  \end{table}

\begin{figure}[H] \centering
  \begin{picture}(420,100)(0,20)
  \put(-10,-445){\includegraphics[width=410pt]{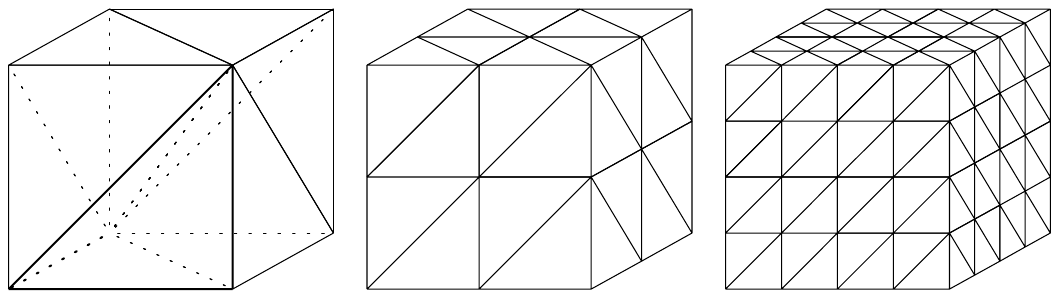}}
  \end{picture}
  \caption{The first three meshes $G_1$--$G_3$ for Table \ref{t3-2} } \label{g3-2}
	\end{figure}

We list the computational results for computing \eqref{s-2} on tetrahedral meshes shown in
   Figure \ref{g3-2}, by the weak Galerkin finite element methods,  in Table \ref{t3-2}.

  \begin{table}[H]
  \caption{ Error profile for computing \eqref{s-2} on Figure \ref{g3-2} meshes.} \label{t3-2}
\begin{center}  
   \begin{tabular}{c|rr|rr|rr}  
 \hline 
$G_i$ & \multicolumn{2}{c|}{$\|\b u-\b u_h\|_0$ $O(h^r)$} & 
        \multicolumn{2}{c|}{$\|\nabla_w(\b u-\b u_h)\|_0$ $O(h^r)$}  & 
   \multicolumn{2}{c}{$\|Q_0p-p_h\|_0$  $O(h^r)$ } \\ \hline  
     & \multicolumn{6}{c}{By the $(P_1,P_1)$-$P_0$ weak Galerkin finite element}  \\ 
 2 &    0.983E-2 &1.80 &    0.239E+0 &0.80 &    0.106E+0 &0.00 \\
 3 &    0.310E-2 &1.67 &    0.127E+0 &0.91 &    0.460E-1 &1.21 \\
 4 &    0.857E-3 &1.85 &    0.653E-1 &0.96 &    0.154E-1 &1.58 \\
 \hline 
     & \multicolumn{6}{c}{By the $(P_2,P_2)$-$P_1$ weak Galerkin finite element}  \\ 
 1 &    0.282E-2 &0.00 &    0.472E-1 &0.00 &    0.116E-1 &0.00 \\
 2 &    0.410E-3 &2.78 &    0.126E-1 &1.91 &    0.203E-2 &2.51 \\
 3 &    0.555E-4 &2.88 &    0.320E-2 &1.97 &    0.356E-3 &2.51 \\
 \hline 
\end{tabular} \end{center}  \end{table}

\begin{figure}[H] \centering
  \begin{picture}(420,100)(0,20)
  \put(-10,-445){\includegraphics[width=410pt]{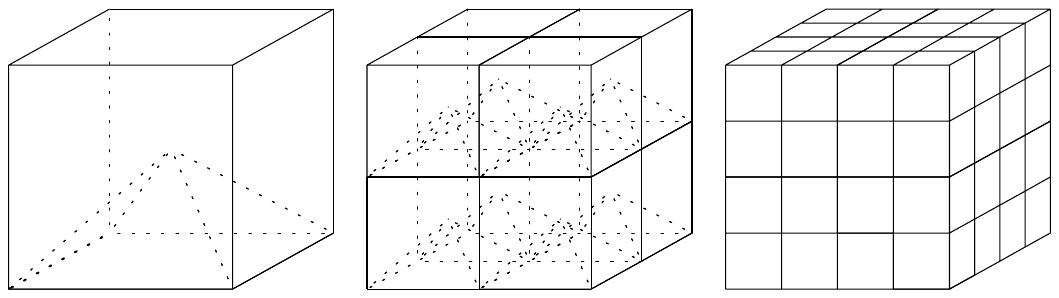}}
  \end{picture}
  \caption{The first three meshes $G_1$--$G_3$ for Table \ref{t3-3} } \label{g3-3}
	\end{figure}

We list the computational results for computing \eqref{s-2} on nonconvex polyhedral meshes shown in
   Figure \ref{g3-3}, by the weak Galerkin finite element methods,  in Table \ref{t3-3}.

  \begin{table}[H]
  \caption{ Error profile for computing \eqref{s-2} on Figure \ref{g3-3} meshes.} \label{t3-3}
\begin{center}  
   \begin{tabular}{c|rr|rr|rr}  
 \hline 
$G_i$ & \multicolumn{2}{c|}{$\|\b u-\b u_h\|_0$ $O(h^r)$} & 
        \multicolumn{2}{c|}{$\|\nabla_w(\b u-\b u_h)\|_0$ $O(h^r)$}  & 
   \multicolumn{2}{c}{$\|Q_0p-p_h\|_0$  $O(h^r)$ } \\ \hline  
     & \multicolumn{6}{c}{By the $(P_1,P_1)$-$P_0$ weak Galerkin finite element}  \\ 
 2 &    0.159E-1 &1.85 &    0.571E+0 &0.52 &    0.228E+0 &0.00 \\
 3 &    0.449E-2 &1.83 &    0.283E+0 &1.01 &    0.891E-1 &1.36 \\
 4 &    0.114E-2 &1.98 &    0.139E+0 &1.03 &    0.270E-1 &1.72 \\
 \hline 
     & \multicolumn{6}{c}{By the $(P_2,P_2)$-$P_1$ weak Galerkin finite element}  \\ 
 2 &    0.891E-3 &2.83 &    0.308E-1 &1.99 &    0.743E-2 &2.48 \\
 3 &    0.119E-3 &2.90 &    0.709E-2 &2.12 &    0.917E-3 &3.02 \\
 4 &    0.156E-4 &2.93 &    0.175E-2 &2.02 &    0.133E-3 &2.79 \\
 \hline 
\end{tabular} \end{center}  \end{table}

\begin{figure}[H] \centering
  \begin{picture}(420,100)(0,20)
  \put(-10,-445){\includegraphics[width=410pt]{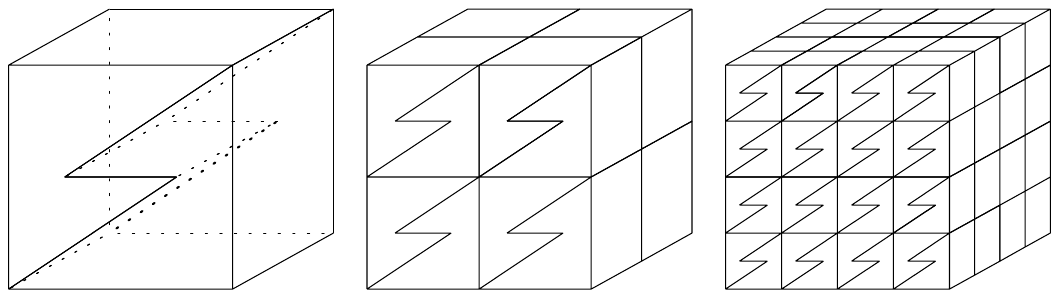}}
  \end{picture}
  \caption{The first three meshes $G_1$--$G_3$ for Table \ref{t3-4} } \label{g3-4}
	\end{figure}

We list the computational results for computing \eqref{s-2} on nonconvex polyhedral  meshes shown in
   Figure \ref{g3-4}, by the weak Galerkin finite element methods,  in Table \ref{t3-4}.

  \begin{table}[H]
  \caption{ Error profile for computing \eqref{s-2} on Figure \ref{g3-4} meshes.} \label{t3-4}
\begin{center}  
   \begin{tabular}{c|rr|rr|rr}  
 \hline 
$G_i$ & \multicolumn{2}{c|}{$\|\b u-\b u_h\|_0$ $O(h^r)$} & 
        \multicolumn{2}{c|}{$\|\nabla_w(\b u-\b u_h)\|_0$ $O(h^r)$}  & 
   \multicolumn{2}{c}{$\|Q_0p-p_h\|_0$  $O(h^r)$ } \\ \hline  
     & \multicolumn{6}{c}{By the $(P_1,P_1)$-$P_0$ weak Galerkin finite element}  \\ 
 2 &    0.172E-1 &1.73 &    0.643E+0 &0.78 &    0.279E+0 &0.00 \\
 3 &    0.493E-2 &1.81 &    0.331E+0 &0.96 &    0.105E+0 &1.41 \\
 4 &    0.127E-2 &1.96 &    0.166E+0 &0.99 &    0.321E-1 &1.70 \\
 \hline 
     & \multicolumn{6}{c}{By the $(P_2,P_2)$-$P_1$ weak Galerkin finite element}  \\ 
 2 &    0.993E-3 &2.65 &    0.384E-1 &2.04 &    0.850E-2 &2.37 \\
 3 &    0.138E-3 &2.84 &    0.935E-2 &2.04 &    0.128E-2 &2.73 \\
 4 &    0.181E-4 &2.93 &    0.231E-2 &2.01 &    0.201E-3 &2.67 \\
 \hline 
     & \multicolumn{6}{c}{By the $(P_3,P_3)$-$P_2$ weak Galerkin finite element}  \\ 
 1 &    0.395E-3 &0.00 &    0.160E-1 &0.00 &    0.539E-2 &0.00 \\
 2 &    0.260E-4 &3.93 &    0.175E-2 &3.19 &    0.325E-3 &4.05 \\
 3 &    0.169E-5 &3.94 &    0.194E-3 &3.18 &    0.216E-4 &3.91 \\
 \hline 
\end{tabular} \end{center}  \end{table}

\end{document}